\documentclass[11pt]{amsart}
\input{preamble.tex}
\usetikzlibrary{patterns,snakes}

\newcommand{\SMor}{\operatorname{SMor}}
\newcommand{\LS}{\EuScript{L}}
\renewcommand{\RS}{\EuScript{R}}

\newcommand{\PB}{\mathbf{P}}
\newcommand{\AB}{\mathbf{A}}
\newcommand{\EB}{\mathbf{E}}

\renewcommand{\SB}{\mathbf{S}}
\newcommand{\XS}{\EuScript{X}}
\newcommand{\XB}{\mathbf{X}}
\newcommand{\YB}{\mathbf{Y}}

\newcommand{\MS}{\EuScript{M}}

\newcommand{\Pretr}{\mathbf{Pretr}}
\renewcommand{\pretr}{\Pretr}
\newcommand{\bra}[1]{\boldsymbol{\langle}#1\boldsymbol{|}}
\newcommand{\ket}[1]{\boldsymbol{|}#1\boldsymbol{\rangle}}

\newcommand{\braket}[2]{\boldsymbol{\langle}#1\boldsymbol{|}#2\boldsymbol{\rangle}}
\newcommand{\braCket}[3]{\left \langle#1\right |#2\left|#3\right\rangle}
\newcommand{\ketbra}[2]{\ket{#1}\bra{#2}}
\newcommand{\Tot}{\operatorname{Tot}}

\newcommand{\Twix}{\mathbf{Twix}}

\newcommand{\dgcat}{\operatorname{dgcat}}
\newcommand{\DGCAT}{\operatorname{DGCAT}}
\newcommand{\Fun}{\operatorname{Fun}}

\begin{document}
\author{Matthew Hogancamp}
\address{Department of Mathematics, Northeastern University, 360 Huntington Ave, Boston,
MA 02115, USA}
\email{m.hogancamp@northeastern.edu}

\title{Envelopes and the bar complex}

\begin{abstract} 
This paper is intended as a reference for some basic theory for dg categories and their bar complexes.  Our modest goal is to carefully record the most important envelope operations can one perform on dg categories (in which one adjoins shifts, finite direct sums, or twists) and the inescapable sign rules that appear when combining these with opposite categories, tensor products, and the bar resolution.  An appendix collects some theory of categorical idempotents that is useful when discussing bar complexes.
\end{abstract}

\maketitle

\setcounter{tocdepth}{1}

\tableofcontents

\section{Introduction}
\label{s:intro}
The word \emph{envelope} (or sometimes \emph{hull}) is used to refer to an idempotent operation that can be performed on a category.  Common examples include the Karoubian envelope, which adjoins the images of all idempotent endomorphisms, the additive envelope  of a $\k$-linear category, which adjoins finite direct sums, and numerous others.  In the context of dg categories we may also adjoin degree shifts of objects, or more general twisted complexes.

Formally speaking, each envelope construction corresponds to an idempotent monad acting on the category of appropriately enriched ($\k$-linear, or dg) categories, with morphisms given by functors up to natural transformation.

\begin{remark}
Keeping track of natural transformations between functors, the collection of (appropriately enriched) categories forms an (appropriately enriched) 2-category, and an envelope operation ought to correspond to an idempotent 2-monad acting on this 2-category.  In the interest of keeping the discussion elementary, we will not focus on the 2-monad structure in this paper, though it is discussed briefly in Remark \ref{rmk:idempotent monad}.
\end{remark}

The main goal of this paper is to describe the most common envelope operations that can be performed on a dg category $\CS$, and to record precisely how these envelopes interact with viarous other well known constructions, such as opposite categories, the tensor product of dg categories, and the bar construction.  We will focus our attention on the following envelope operations which can be performed on any dg category $\CS$:
\begin{enumerate}
\item the suspended envelope $\SB(\CS)$, in which we formally adjoin shifts (\S \ref{ss:shifts}),
\item the additive envelope $\AB(\CS)$, in which we formally adjoin finite direct sums (\S \ref{ss:finite sums}),
\item the twisted envelope $\Tw(\CS)$, in which we formally adjoin twists of objects (\S \ref{ss:twists})
\item the pretriangulated envelope $\pretr(\CS)$, in which we formally adjoin all mapping cones (\S \ref{ss:twisted complexes}).
\end{enumerate}
Combining the first three envelopes above yields the category $\Twix(\CS):=\Tw\circ \AB\circ \SB(\CS)$ consisting of \emph{finite twisted complexes over $\CS$}.  Then $\pretr(\AS)$ is a full subcategory of $\Twix(\CS)$ consisting of \emph{one-sided twisted complexes} (\S \ref{ss:twisted complexes}).

\begin{remark}
The category of twisted complexes was introduced by Bondal and Kapranov \cite{BonKap90}, where it was called the pretriangulated envelope of $\CS$.  However, the category of arbitrary twisted complexes is not a well-behaved invariant of $\CS$.  For instance one can find an example of a dg category $\CS$ which is quasi-equivalent to zero, but for which $\Tw(\CS)$ is not quasi-equivalent to zero (see Example \ref{ex:Tw is not invt}).

In contrast, the category of finite one-sided twisted complexes is always derived Morita equivalent to $\CS$.  For this reason, the modern tradition is to reserve the term pretriangulated envelope for the category of one-sided twisted complexes, as we have done.  For the standard references on homotopy theory and Morita theory of dg categories, see \cite{KellerDeriving,KellerICM,ToenDg,ToenMorita,DrinfeldQuotient}.
\end{remark}

\begin{remark}
If all the hom complexes in $\CS$ are supported in non-positive cohomological degrees, then every twisted complex is one-sided, so $\Twix(\CS)$ and $\pretr(\CS)$ coincide (see Remark \ref{rmk:negatively graded}).
\end{remark}

\begin{remark}
If $\AS$ is an ordinary 
additive category, then we may regard $\AS$ as a dg category with trivial differential and grading, and $\pretr(\AS)=\Twix(\AS)=\Ch^b(\AS)$ is the usual dg category of complexes over $\AS$.
\end{remark}

When working with twisted complexes, there are some unavoidable sign rules that enter into many constructions, and one of the main purposes of this paper is to provide a reference for some of the most common of these.  To describe, let $\EB$ be one of the envelope operations considered above.  
\begin{itemize}
\item  in \S \ref{ss:opp} we discuss an isomorphism $\EB(\CS^{\op})\cong \EB(\CS)^{\op}$, which is responsible for extending a contravariant functor from $\CS$ to $\DS$ a contravariant functor from $\EB(\CS)$ to $\EB(\DS)$.
\item in \S \ref{ss:tensor} we discuss a fully faithful dg functor from $\EB(\CS_1)\otimes \EB(\CS_2)\to \EB(\CS_1\otimes \CS_2)$, which is responsible for extending a multilinear functor $\CS_1\times \cdots \times \CS_r\to \DS$ to a multilinear functor $\EB(\CS_1)\times \cdots \times \EB(\CS_r)\to \EB(\DS)$.
\end{itemize}

These types of constructions are essential when considering, for instance, monoidal structures and duals for categories of the form $\pretr(\CS)$.

In \S \ref{s:bimodules} we give a precise relationship between the bar complex of $\CS$ and its envelopes. This firmly establishes a fact stated imprecisely and without proof in \cite[\S 5.3]{GHW}.  The bar complex of a dg category gives a projective resolution of the identity bimodule, and our main result in \S \ref{s:morita} is that $\CS$ is derived Morita equivalent to its envelopes $\EB(\CS)$, for $\EB\in \{\SB,\AS,\pretr\}$.

\begin{remark}
This paper offers something of a different perspective than the standard references on dg categories, in a few ways.  First, we do not discuss model category structures at all, since this is already quite thoroughly studied.  Thanks to To\"en \cite{ToenMorita}, we know that weak equivalence of dg categories (with respect to the appropriate model structure) agrees with derived Morita equivalence, which can be approached with a combination of usual Morita theory and some bar complex trickery.

Secondly, we prioritize formulas.  One can show by general nonsense that $\CS$ is derived Morita equivalent to $\pretr(\CS)$, but we will achieve this result by an explicit chain map relating the bar complexes of $\CS$ and $\pretr(\CS)$.  This same formula was stated without careful bookkeeping and signs in \cite[\S 5.3]{GHW}.  Perhaps the main theme of this paper is: we have recorded (and, hopefully, motivated) as many of the irritating sign rules that we could think of, so future work does not have to.  We have tried to keep the prerequisites as minimal and the paper as self-contained as possible.  We do rely on some theory of counital idempotents \cite{HogancampIdempotent,HogancampConstructing} in order to show that the chain map we construct is a homotopy equivalence.  Since these papers do not directly apply to our dg setting, we have decided to include the relevant theory in Appendix \ref{s:counital idempotents} for the reader's convenience.
\end{remark}

\section{DG categories}
\label{s:setup}

\subsection{Graded $\k$-modules}
\label{ss:graded modules}
Throughout the paper, fix a commutative ring $\k$.  All graded $\k$-modules will be graded by an abelian group $\Gamma$.  We will also assume given a symmetric bilinear form $\ip{\ , \ }\colon \Gamma\rightarrow \Z/2$, and an element $\iota\in \Gamma$ with $\ip{\iota,\iota}=1$, which is to be the degree of all differentials.   For concreteness, the reader can take $\Gamma=\Z$, $\ip{i,j}=ij$ (mod 2), and $\iota=1$.

A $\Gamma$-graded $\k$-module will be written as $M=\{M^i\}_{i\in \Gamma}$ with $M^i\in \Mod{\k}$ being referred to as the component of $M$ in degree $i$. If $m\in M^i$ then we say $m$ is homogeneous of degree $i$ and write $|m|=i$ or $\deg(m) =i$.  We may also say that $m$ has \emph{weight} $\wt(m) = q^i$, with $q$ a formal variable.   We write $m\in M$ if $m\in M^i$ for some $i\in \Gamma$; that is to say, we only ever speak of \emph{homogeneous elements} of graded modules.

The role of $\ip{\ , \ }$ is that it defines the symmetric monoidal structure on $\Gamma$-graded $\k$-modules.  Namely, we have the \emph{tensor product} $M\otimes N$ defined by
\[
(M\otimes N)^k :=\bigoplus_{i+j=k} M^i\otimes N^j,
\]
with \emph{symmetric structure}
$
\tau_{M,N}\colon M\otimes N\xrightarrow{\cong} N\otimes M
$
defined componentwise by
\[
(\tau_{M,N})|_{M^i\otimes N^j} \colon M^i\otimes N^j \xrightarrow{(-1)^{\ip{i,j}}} N^j\otimes M^i.
\]
We also have the $\Gamma$-graded $\k$-module of homs (the \emph{internal hom}), defined by
\begin{equation}\label{eq:graded hom}
\Hom_\k(M,N)=\{\Hom_\k^l(M,N)\}_{l\in \Gamma} \ ,\qquad  \Hom_\k^l(M,N) := \prod_{i} \Hom_\k(M^i,N^{i+l}).
\end{equation}
An element $f\in \Hom_\k^l(M,N)$ will be referred to as a $\k$-linear map of degree $l\in \Gamma$.

\begin{definition}
Let $\gMod{\Gamma}{\k}$ denote the category of $\Gamma$-graded $\k$-modules and degree zero $\k$-linear maps. A \emph{$\Gamma$-graded $\k$-linear category} will mean a category enriched in $\gMod{\Gamma}{\k}$.  The internal hom \eqref{eq:graded hom} gives $\gMod{\Gamma}{\k}$ the structure of a $\Gamma$-graded category, denoted $(\gMod{\Gamma}{\k})^{\mathrm{en}}$.
\end{definition}

Given $f\in \Hom^k_\k(M,M')$ and $g\in \Hom^l_\k(N,N')$ we define $f\otimes g\in \Hom^{k+l}_{\k}(M\otimes M',N\otimes N')$ by the formula (the Koszul sign rule)
\begin{equation}\label{eq:koszul tensor}
(f\otimes g)(m\otimes n) = (-1)^{\ip{|g|,|m|}} f(m)\otimes g(n).
\end{equation}
This tensor product extends the (symmetric) monoidal structure from $\gMod{\Gamma}{\k}$ to $(\gMod{\Gamma}{\k})^{\mathrm{en}}$.  One can easily check that the tensor product of morphisms of nonzero degree satisfies
\begin{equation}\label{eq:interchange law}
(f\otimes g)\circ (f'\otimes g') = (-1)^{\ip{|f|,|g'|}}(f\circ f')\otimes (g\circ g').
\end{equation}

\subsection{Grading shift functors}
\label{ss:k shifts}
Retain the setup from \S \ref{ss:graded modules}.
For each $j\in \Gamma$ we define the grading shift functor $q^j(-)\colon \gMod{\Gamma}{\k}\to \gMod{\Gamma}{\k}$ by
\begin{equation}\label{eq:shift of M}
(q^jM)^i = M^{i-j}  \qquad q^jf= (-1)^{\ip{j,|f|}} f.
\end{equation}
We can think of the grading shift functor as implemented by tensoring with $q^j\k$ \emph{on the left}:
\[
q^jM \buildrel\cong\over\rightarrow q^j\k \otimes M \, , \qquad m\mapsto 1\otimes m.
\]
We can also define a grading shift functor $(-)q^j\colon \gMod{\Gamma}{\k}\to \gMod{\Gamma}{\k}$ which is modeled on tensoring with $q^j\k$ \emph{on the right}; this acts on morphisms with no sign: 
\begin{align*}
(Mq^j)^i = M^{i-j}  \qquad fq^j= f\\
Mq^j \buildrel\cong\over\rightarrow M\otimes q^j \, ,\qquad m\mapsto m\otimes 1
\end{align*}
These are isomorphic using the symmetric structure:
\[
q^jM\cong Mq^j \ , \qquad m\mapsto (-1)^{\ip{j,|m|}} m.
\]

\subsection{Turning on differentials}
\label{ss:differentials}
Retain the setup from \S \ref{ss:graded modules}.  Recall that we have a chosen element $\iota\in \Gamma$ such that $\ip{\iota,\iota}= 1$.  A \emph{differential $\Gamma$-graded $\k$-module} (or \emph{dg module} for short) is a pair $(M,\d_M)$ where $M\in \gMod{\Gamma}{\k}$ and $\d_M\in \End^\iota_\k(M)$ satisfies $\d_M\circ \d_M=0$.  Let $\dgMod{\Gamma}{\k}$ denote the category of dg modules over $\k$ (graded by $\Gamma$), with morphisms given by degree zero chain maps.

The tensor product of dg modules is defined by
\[
(M,\d_M)\otimes (N,\d_N):= (M\otimes N, \d_{M\otimes N}:=\d_M\otimes \id_N + \id_M\otimes \d_N),
\]
where we are using \eqref{eq:koszul tensor} to define the tensor product of morphisms.  This is a differential because (using \eqref{eq:interchange law} to compose tensor products of morphisms) we have
\[
(\d_{M\otimes N})^2 = 0 + \d_M\otimes \d_N + (-1)^{\ip{\iota,\iota}}\d_M\otimes \d_N + 0,
\]
which is zero because $\ip{\iota,\iota}=1$.  The internal hom of dg modules is defined by
\[
\Hom_\k\Big((M,\d_M),(N,\d_N)\Big) = \Big(\Hom_\k(M,N) \: , \:  \d_{\Hom(M,N)}(f):=\d_N\circ f - (-1)^{\ip{\iota,|f|}} f\circ \d_M\Big).
\]
We will frequently abbreviate, writing $\d_{\Hom(M,N)}(f)$ simply as $d_\k(f)$.
The grading shift extends to dg modules by declaring
\[
q^j (M,\d_M) = (q^jM , q^j \d_M)
\]
Note that this will multiply the differential by a sign $(-1)^{\ip{j,\iota}}$, according to \eqref{eq:shift of M}. 

\begin{example}\label{ex:classical setup}
In the classical setup of dg categories we have $\Gamma=\Z$ with pairing $\ip{i,j} = ij$ (mod 2), and $\iota=1$ (or $\iota=-1$ if one prefers).  The grading shift $q^i M$ is conventionally denoted $M[-i]$.
\end{example}

\subsection{Dg category basics}
\label{ss:dg basics}

\begin{definition}\label{def:dg cat}
A \emph{differential $\Gamma$-graded category} (or \emph{dg category} for short) will mean a category enriched in $\dgMod{\Gamma}{\k}$. 
\end{definition}
What this means is that the hom spaces in $\CS$ are dg $\k$-modules, and composition defines degree zero chain maps
\begin{equation}\label{eq:composition}
\Hom_\CS(Y,X)\otimes \Hom_\CS(Z,Y)\to \Hom_\CS(Z,X),
\end{equation}
with tensor product taken in $\dgMod{\Gamma}{\k}$. 
\begin{remark}
Any $\k$-linear category may be regarded a dg category with trivial grading and differential.
\end{remark}

We will often adopt the bra-ket notation from physics literature to denote the hom complexes in $\CS$. If $\CS$ is a dg category then we will denote hom complexes in $\CS$ in a ``right-to-left'' fashion by
\begin{equation}\label{eq:XCY}
\braCket{X}{\CS}{Y}:=\Hom_\CS(Y,X)
\end{equation}
When the category $\CS$ is understood, it may be omitted from the notation as in $\braket{X}{Y}=\braCket{X}{\CS}{Y}$.  The bra-ket notation will ultimately play a major role in our discussion of bimodules \S \ref{s:bimodules}, but also the freedom to sometimes write hom complexes in a right-to-left fashion elucidates certain formulas.  Most obviously, composition of morphisms defines a chain map
\[
\begin{tikzpicture}
\node (a) at (0,0) {$\braCket{X}{\CS}{Y} \otimes \braCket{Y}{\CS}{Z}$};
\node (b) at (4,0) {$\braCket{X}{\CS}{Z}$};
\node (c) at (0,-.5) {$f\otimes g$};
\node (d) at (4,-.5) {$f\circ g$};
\path[-stealth]
(a) edge node {} (b);
\path[|-stealth]
(c) edge node {} (d);
\end{tikzpicture}
\]

The differential in the hom complexes of $\CS$ will be denoted $d_\CS$ or $d$. A morphism $f\in\Hom_\CS(X,Y)$ in a dg category is \emph{closed} if $d_\CS(f)=0$, and \emph{exact} if $f=d_\CS(h)$ for some $h\in \Hom_\CS(X,Y)$.  Objects $X,Y\in \CS$ are \emph{isomorphic}, written $X\cong Y$, if there is a degree zero closed invertible morphism $f\colon X\rightarrow Y$ in $\CS$ (if $f$ is not closed, then instead we say that $Y$ is a twist of $X$; see below).

We say that two closed morphisms $f,g\in \Hom_\CS(X,Y)$ are \emph{homotopic}, written $f\simeq g$, if $|f|=|g|$ and $f-g$ is exact.  We write $Z^0(\CS)$ for the category with the same objects as $\CS$, and morphisms the degree zero closed morphisms.  We write $H^0(\CS)$ for the \emph{cohomology category}, i.e.~ the category with the same objects as $\CS$ and morphisms the degree zero closed morphisms modulo exact morphisms.  We say that objects $X,Y$ are homotopy equivalent, also written $X\simeq Y$, if $X$ and $Y$ are isomorphic in $H^0(\CS)$.  An object $X$ is \emph{contractible} if $X\simeq 0$.  If $X\simeq 0$ then a \emph{contracting homotopy} for $X$ is an element $h\in \End^{-\iota}(X)$ such that $d_\CS(h)=\id_X$.

A \emph{dg functor} $F\colon \CS\rightarrow \DS$ is a mapping on objects $F\colon \Obj(\CS)\rightarrow \Obj(\DS)$ and on morphisms $F\colon \Hom_\CS(X,X')\rightarrow \Hom_\DS(F(X),F(X'))$ such that $F(\id_X)=\id_{F(X)}$ and $F(f\circ f')=F(f)\circ F(f')$ for all $X$ and all composable morphisms $f,f'$. Given dg functors $F,G\colon \CS\rightarrow \DS$, the \emph{dg $\k$-module of natural transformations}, denoted $\Hom_\k(F,G)$ is defined as follows.  An element of $\Hom_\k^l(F,G)$ is a family of morphisms $(\a_X\in \Hom_\DS^l(F(X),G(X)))_{X\in \Obj(\CS)}$ such that
\[
G(f)\circ \a_X = (-1)^{\ip{l,|f|}}\a_{X'}\circ F(f) \text{ for all } f\in \Hom_\CS(X,X').
\]
The differential of $\a = (\a_X)_X$ is defined to be $(d_{\DS}(\a_X))_X$.

Now, let $\CS$ and $\DS$ be dg categories, and let $F\colon \CS\to \DS$ be a dg functor.  We say $F$ is \emph{quasi-essentially surjective} if each $Y\in\Obj(\DS)$ is homotopy equivalent to $F(X)$ for some $X\in \Obj(\CS)$, and \emph{quasi-fully faithful} if each mapping on hom complexes $\Hom_\CS(X,X')\to \Hom_\DS(F(X),F(X'))$ is an isomorphism in homology.  We say $F$ is a \emph{quasi-equivalence} if it is quasi-essentially surjective and quasi-fully faithful.

\section{The envelopes}
\label{s:closures}

%

\subsection{Adjoining shifts}
\label{ss:shifts}
Let $\CS$ be a dg category, and let $X\in \CS$ and $i\in \Gamma$ be given.  We say that $Y\in \CS$ \emph{realizes the $i$-shift of $X$}, written $Y\cong q^iX$, if there is a closed, invertible, degree $i$ map $\phi\colon X\to Y$.  We say that $\CS$ is \emph{suspended} if for every $X\in \Obj(\CS)$ and every $i\in \Gamma$, there exists an object realizing $q^iX$. Note that the pair $(Y,\phi)$ realizing $q^i X$ is unique when it exists, since if $(Y',\phi')$ is another pair realizing $q^iX$, then $\phi'\circ \phi\inv\colon Y\to Y'$ is an isomorphism. 

\begin{remark}
If $Y$ realizes $q^i X$ in $\CS$, then for any dg functor $F\colon \CS\to \DS$, the shift $q^i F(X)$ is realized by $F(Y)$ in $\DS$, with structure map $F(\phi)\colon F(X)\to F(Y)$.  In other words, every dg functor commutes with shifts:
\[
F(q^iX)\cong q^iF(X).
\]
\end{remark}

Every dg category embeds fully faithfully in a suspended dg category.

\begin{definition}\label{def:suspended envelope}
Let $\SB(\CS)$, the \emph{suspended envelope of $\CS$}, denote the category with objects formal expressions $q^iX$ with $i\in \Gamma$ and $X$, and hom complexes
\begin{equation}\label{eq:shifted homs}
\begin{tikzpicture}[baseline=-.1cm]
\node at (0,0) {$\Hom_{\SB(\CS)}^l(q^iX,q^jY)$};
\node at (1.8,0) {$:=$};
\node at (3.4,0) {$\Hom_\CS^{l+i-j}(X,Y)$};
\node at (1,-.75) {$f^j_i$};
\node at (1.8,-.75) {$\substack{\longmapsto \\ \longmapsfrom}$};
\node at (2.5, -.75) {$f$};
\end{tikzpicture}
\ , \qquad \quad d_{\SB(\CS)}(f^j_i) := (-1)^{\ip{\iota,j}} d_\CS(f)^j_i
\end{equation}
Here we are adopting the convention that for a given morphism $f\in \Hom_\CS^k(X,Y)$ and a pair of elements $i,j\in \Gamma$, the corresponding element of $\Hom_{\SB(\CS)}^{k+j-i}(q^iX,q^jY)$ shall be denoted by $f^j_i$. 

For each $X$, denote the object $q^0X\in \SB(\CS)$ simply by $X$ and the morphism $f^0_0$ simply by $f$. Given $i\in \Gamma$, let $\phi_{X,i}=(\id_X)^i_0\in \Hom^i_{\SB(\CS)}(X,q^iX)$.
\end{definition}

\begin{remark}
For each $X\in \CS$ and each $i\in \Gamma$, the map $\phi_{X,i}=(\id_X)^i_0$ is closed and invertible, with inverse $\phi_{X,i}\inv = (\id_X)^0_{i}$, hence the formal expression $q^i X$ does indeed realize the shift of $X$ in $\SB(\CS)$.
\end{remark}

\begin{remark}
Using the canonical structure maps $\phi_{X,i}$, the identification $\Hom_\CS^k(X,Y) \buildrel =\over\rightarrow \Hom_{\SB(\CS)}^{k+j-i}(q^iX,q^jY)$ from \eqref{eq:shifted homs} is implemented by $f\mapsto \phi_{Y,j}\circ f\circ \phi_{X,i}\inv$, i.e.
\begin{equation}\label{eq:fij}
f^{j}_{i} = \phi_{Y,j}\circ f\circ \phi_{X,i}\inv.
\end{equation}
The sign rule on the differential $d_{\SB(\CS)}(f^j_i)$ from \eqref{eq:shifted homs} is exactly what is required to ensure compatibility of \eqref{eq:fij} with the Leibniz rule.
\end{remark}

\begin{remark}
In terms of the ``bra-ket'' notation \eqref{eq:XCY}, equation \eqref{eq:shifted homs} can also be written
\[
\braCket{q^jY}{\SB(\CS)}{q^iX} = q^j \braCket{Y}{\CS}{X} q^{-i}.
\]
(note that the signs on the differentials agree).
\end{remark}

\begin{remark}\label{rmk:SB on functors} 
Any dg functor $F\colon \CS\to \DS$ induces a dg functor on suspended envelopes $\SB(F)\colon \SB(\CS)\to \SB(\DS)$ by the formulas
\[
\SB(F)(q^iX) := q^i F(X) \, , \qquad \SB(F)(f^j_i) := F(f)^j_i.
\]
In fact $\SB(F)$ is the unique extension of $F$ such that $\SB(F)(\phi_{X,i})=\phi_{F(X),i}$.
\end{remark}

\begin{proposition}\label{prop:suspended envelope}
We have the following:
\begin{enumerate}
\item  there is a fully faithful dg functor $\eta_\CS\colon \CS\to \SB(\CS)$ sending $X\mapsto q^0X$ and $f\mapsto f^0_0$.
\item there is an equivalence of dg categories $\mu_\CS\colon \SB(\SB(\CS))\to \SB(\CS)$ sending $q^i(q^{i'}X))\mapsto q^{i+i'}X$ and $(f^{i'}_{j'})^i_j\mapsto f^{i+i'}_{j+j'}$.
\item the pair $(\mu_\CS,\eta_\CS)$ satisfies appropriate associativity and unit axioms up to isomorphism.
\item $\CS$ is suspended iff $\eta_\CS\colon \CS\to \SB(\CS)$.
\end{enumerate} 
\end{proposition}
\begin{proof}
Straightforward exercise.
\end{proof}



\begin{remark}\label{rmk:idempotent monad}
Remark \ref{rmk:SB on functors} says that $\SB$ may be regarded as a functor $\SB\colon \dgcat \to \dgcat$, where $\dgcat$ is the category of dg categories and isomorphism classes of functors between them.  Property (3) from Proposition \ref{prop:suspended envelope} say that in fact $\SB$ has the structure of a \emph{monad} acting on $\dgcat$.   Further, this monad is idempotent by (2), since $\mu_\CS$ is an isomorphism in $\dgcat$.

One can promote all these statements one ``category-level'' higher.  That is to say, we may endow $\SB$ with the structure of an idempotent 2-monad acting on the 2-category $\DGCAT$ of dg categories, dg functors, and dg natural tranformations. We won't discuss this more in the present paper.
\end{remark}

\begin{remark}
If $\CS$ is suspended, then choosing an inverse equivalence to $\eta_\CS\colon \CS\to \SB(\CS)$ gives $\SB(\CS)\to \CS$, which we think of as a functorial realization of all shifts of objects in $X$.

One can see that $\SB(\CS)$ is suspended, with a functorial realization of shifts provided by $\mu_\CS\colon \SB(\SB(\CS))\to \SB(\CS)$.  In other words, given an object $q^{i'}X\in \SB(\CS)$, its shift $q^i(q^{i'}X)$ is realized in $\SB(\CS)$ by $q^{i+i'}X$.
\end{remark}

We conclude this section with a description of how shifting by $i\in \Gamma$ looks as a dg functor.

\begin{lemma}[Shift functors]
\label{lemma:shift}
For each $k\in \Gamma$ there is a well-defined dg functor $q^k\colon \SB(\CS)\rightarrow \SB(\CS)$ defined on objects by $q^iX\mapsto q^{i+k}X$, and on morphisms $f^j_i\in \Hom_{\SB(\CS)}(q^iX,q^jY)$ by
\[
q^k f^j_i:= (-1)^{\ip{k, |f|+j-i}} f^{j+k}_{i+k}.
\]
\end{lemma}
\begin{proof}
Exercise.
\end{proof}

%
%
%


\subsection{Adjoining finite direct sums}
\label{ss:finite sums}
Let $A$ be a finite set, and suppose we have an $A$-indexed family of objects $X_a\in \CS$.  We say that $Y\in \CS$ \emph{realizes the direct sum of the $X_a$}, written $Y\cong \bigoplus_{a\in A} X_a$, if $Y$ comes equipped with $A$-indexed families of closed, degree zero maps $\sigma_a\in \Hom_\CS(X_a,Y)$ and $\pi_a\in \Hom_\CS(Y,X_a)$ such that
\[
\pi_a\circ \sigma_b  = \begin{cases} \id_{X_a} & \text{ if $a=b$} \\ 0 &\text{ otherwise}\end{cases} \, \qquad \text{and} \qquad
\sum_{a\in A} \sigma_a\circ \pi_a = \id_Y.
\]
If $\{(Y,\{\sigma_a\}_{a}, \{\pi_a\}_{a}\}$ and $\{(Y',\{\sigma_a'\}_{a}, \{\pi_a'\}_{a}\}$ are two realizations of $\bigoplus_{a}X_a$ then $\phi=\sum_{a} \sigma_a'\circ \pi_a$ is an isomorphism (with two-sided inverse $\phi\inv = \sum_a \sigma_a\circ \pi_a'$).  Thus, a realization of $\bigoplus_{a\in A}X_a$ is unique if it exists.

\begin{remark}
Note that finite direct sums in $\CS$ are the same as finite direct sums, or bi-products, in $Z^0(\CS)$.
\end{remark}

\begin{remark}
Dg functors preserve finite direct sums, because they preserve the equations which charactrize them.
\end{remark}
We say that $\CS$ is \emph{additive} if every finite direct sum is realized in $\CS$.  Every dg category embeds fully faithfully into an additive dg category.

\begin{definition}\label{def:additive envelope}
Let $\AB(\CS)$, the \emph{additive envelope of $\CS$}, be the category with objects formal expressions $\bigoplus_{a\in A}X_a$ with $A$ a finite set and $X_a\in \CS$, and morphisms
\[
\Hom_{\AB(\CS)}\left(\YB,\XB\right) :=\bigoplus_{(a,b)\in A\times B} \Hom_\CS(Y_b,X_a)
\]
or, using the right-to-left ``bra-ket'' notation from \eqref{eq:XCY}
\[
\braCket{\XB}{\AB(\CS)}{\YB} = \bigoplus_{(a,b)\in A\times B} \braCket{X_a}{\CS}{Y_b}.
\]
where $\XB=\bigoplus_{a\in A} X_a$ and $\YB=\bigoplus_{b\in B} Y_b$.
An element of the hom complex above will be written as an $A\times B$ matrix $(f_{a,b})$ with $f_{a,b}\in \Hom_\CS(Y_b,X_a)$.  The ordering of indices is chosen so that composition of morphisms is compatible with matrix multiplication:
\[
(f\circ g)_{a,c} = \sum_{b\in B} f_{a,b}\circ g_{b,c}
\]
If $A=\{a\}$ is a singleton and $X_a=X$, then we will write the object $\bigoplus_{a\in \{a\}} X$ simply as $X$. 
\end{definition}

\begin{remark}Any dg functor $F\colon \CS\to \DS$ induces a dg functor on additive envelopes $\AB(F)\colon \AB(\CS)\to \AB(\DS)$ as follows.  Given objects $\XB=\bigoplus_{a\in A} X_a$ and $\YB=\bigoplus_{b\in B} Y_b$ and a morphism $f\in \Hom_{\AB(\CS)}(\YB,\XB)$, then we define
\[
\AB(F)\Big(\bigoplus_{a\in A} X_a\Big) := \bigoplus_{a\in A} F(X_a) \, , \qquad \AB(F)(f)_{a,b} := F(f_{a,b}).
\]
\end{remark}

We have an analogue of Proposition \ref{prop:suspended envelope} as well.

\begin{proposition}\label{prop:additive envelope}
We have the following:
\begin{enumerate}
\item  there is a fully faithful dg functor $\eta_\CS\colon \CS\to \AB(\CS)$.
\item there is a equivalence of dg categories $\mu_\CS\colon \AB(\AB(\CS))\to \AB(\CS)$.
\item the pair $(\mu_\CS,\eta_\CS)$ satisfies appropriate associativity and unit axioms up to isomorphism.
\item $\CS$ is additive iff $\eta_\CS\colon \CS\to \AB(\CS)$ is an equivalence of dg categories.
\end{enumerate} 
\end{proposition}
\begin{proof}
Straightforward exercise.
\end{proof}

%
\subsection{Adjoining twists}
\label{ss:twists}

We generally refer to the process of ``turning on differentials'' as \emph{twisting}.  For dg modules this looks as follows. If $M$ is a dg $\k$-module with differential $\d_M$, and $\a\in \End^1_\k(M)$ satisfies the Maurer-Cartan equation
\[
[\d_M,\a]+\a^2=0 \qquad (\text{equivalently $(\d_M+\a)^2=0$}),
\]
then we may consider the dg $\k$-module $M$ with \emph{twisted differential}, denoted
\[
\tw_\a(M) := (M,\d_M+\a).
\]

For more general dg categories, we say that $Y$ is a \emph{twist of $X$} if there is a degree zero invertibe (but not necessarily closed) morphism $\psi\in \Hom_\CS(X,Y)$.  In this situation, we construct an endomorphism $\a\in \End^\iota_\CS(X)$ by $\a:=\psi\inv\circ d_\CS(\psi)$, i.e.~$d_\CS(\psi) = \psi\circ \a$.  The fact that $d_\CS^2(\psi)=0$ is equivalent to the Maurer-Cartan equation
\begin{equation}\label{eq:MC}
d_\CS(\a)+\a^2 = 0.
\end{equation}
An endomorphism $\a\in \End^{\iota}_\CS(X)$ satisfying \eqref{eq:MC} will be called a \emph{Maurer--Cartan} endomorphism of $X$, and we will say that \emph{$Y$ realizes the twist of $X$ by $\a$}, written $Y\cong \tw_\a(X)$, if $Y$ is equipped with a degree zero invertible morphism $\psi \colon X\to Y$ such that $d_\CS(\psi)=\psi\circ \a$.  We say that $\CS$ is \emph{twisted} if all twists of all objects are realized in $\CS$.

The following establishes that a twist of $X$ is uniquely determined by its associated Maurer-Cartan endomorphism.
\begin{lemma}\label{lemma:uniqueness of twists}
If $(Y,\psi)$ and $(Y',\psi')$ both realize the twist $\tw_\a(X)$, then $\psi'\circ \psi\inv\colon Y\to Y'$ is an isomorphism.
\end{lemma}
\begin{proof}
First, observe that $d(\psi\inv) = - \psi\inv \circ d(\psi)\circ \psi\inv$, which equals $-\a\circ \psi\inv$. Then by the Leibniz rule
\[
d(\psi'\circ \psi\inv)  = (\psi'\circ \a)\circ \psi\inv - \psi'\circ (\a\circ \psi\inv) = 0
\]
So $\psi'\circ \psi$ is closed (in addition to being degree zero and invertible), hence is an isomorphism.
\end{proof}

\begin{remark}
Dg functors commute with twists: if $Y\cong \tw_\a(X)$, then for any dg functor $\CS\to  \DS$, applying $F$ to the equation $d_\CS(\psi) = \psi\circ \a$ shows that $F(Y)\cong \tw_{F(\a)}(F(X))$.
\end{remark}

Every dg category embeds fully faithfully in a twisted dg category, as we will see below.  We will adopt notational conventions that are similar in spirit to those used in \S \ref{ss:shifts} for the suspended envelope. 
\begin{definition}\label{def:twisted envelope}
Let $\Tw(\CS)$, the \emph{twisted envelope of $\CS$}, be the dg category with objects formal expressions $\tw_\a(X)$ with $X\in \CS$ and $\a\in \End^{\iota}_\CS(X)$ an endomorphism satisfying \eqref{eq:MC} and hom complexes defined by
\begin{equation}\label{eq:twisted homs}
\begin{tikzpicture}[baseline=-.1cm]
\node at (0,0) {$\Hom_{\Tw(\CS)}^k(\tw_\a(X),\tw_\b(Y))$};
\node at (2.5,0) {$:=$};
\node at (3.8,0) {$\Hom_\CS^{k}(X,Y)$};
\node at (1.7,-.75) {$f$};
\node at (2.5,-.75) {$\substack{\longmapsto \\ \longmapsfrom}$};
\node at (3.3, -.75) {$f$};
\end{tikzpicture}
\ , \qquad d_{\Tw(\CS)}(f) := d_\CS(f) + \b\circ f - (-1)^{\ip{\iota, |f|}} f\circ \a.
\end{equation}
When we want to distinguish between elements of $\Hom_\CS(X,Y)$ and $\Hom_{\Tw(\CS)}(\tw_\a(X),\tw_\b(Y))$, we may sometimes write elements of the latter as $f^\b_\a$.  

For each $X$, denote the object $\tw_0(X)\in \Tw(\CS)$ simply by $X$. 
\end{definition}


\begin{remark}
We clearly have
\[
d_{\Tw(\CS)}(\id^\a_0) = (\a)^\a_0 = (\id^\a_0)\circ \a
\]
so indeed the formal expression $\tw_\a(X)$ realizes the twist of $X$ by $\a$ in $\Tw(\CS)$ with structure map given by $\psi_{X,\a}=\id^\a_0$.  Further, the formula defining the differential of $f^\b_\a$ is exactly what is required to be consistent with the Leibniz rule applied to $f^\b_\a = \id^\b_0\circ f\circ \id^0_\a$.
\end{remark}

%

\begin{remark}\label{rmk:twist on functor}
Any dg functor $F\colon \CS\to \DS$ induces a dg functor $\Tw(F)\colon \Tw(\CS)\to \Tw(\DS)$ by the formulas
\[
\Tw(F)(\tw_\a(X)) := \tw_{F(\a)}(F(X)) \, ,\qquad \Tw(F)(f) = F(f)
\]
\end{remark}
%

We have an analogue of Propositions \ref{prop:suspended envelope} and \ref{prop:additive envelope} as well.

\begin{proposition}\label{prop:twisted envelope}
We have the following:
\begin{enumerate}
\item  there is a fully faithful dg functor $\eta_\CS\colon \CS\to \Tw(\CS)$ sending $X\mapsto \tw_0(X)$ and $f\mapsto f$.
\item there is a equivalence of dg categories $\mu_\CS\colon \Tw(\Tw(\CS))\to \Tw(\CS)$ sending $\tw_\a(\tw_{\a'}(X))\mapsto \tw_{\a+\a'}(X)$ and $f\mapsto f$.
\item the pair $(\mu_\CS,\eta_\CS)$ satisfies appropriate associativity and unit axioms up to isomorphism.
\item if $\CS$ is twisted iff $\eta_\CS\colon \CS\to \Tw(\CS)$ is an equivalence of dg categories, and conversely.
\end{enumerate} 
\end{proposition}
\begin{proof}
Straightforward exercise.  Let us make one comment here, namely  that the Maurer-Cartan equation for an endomorphism $\a\in \End_{\Tw(\CS)}(\tw_{\a'}(X))$ reads
\[
\Big(d_\CS(\a)+\a'\circ \a + \a\circ \a'\Big) + \a^2 = 0
\]
Given that $\a'$ is a Maurer-Cartan endomorphism of $X$ in $\CS$, the above is equivalent to the statement that $d_\CS(\a+\a')+(\a+\a')^2=0$.  In other words, $\tw_\a(\tw_{\a'}(X))$ is a well-defined object of $\Tw(\Tw(\CS))$ if and only if both $\tw_{\a'}(X)$ and $\tw_{\a+\a'}(X)$ are well-defined objects of $\Tw(\CS)$.  This is the key fact which shows that $\Tw(\CS)$ is twisted, and enables one to construct the dg functor $\mu_\CS$ as in the statement.
\end{proof}

\subsection{Adjoining twisted complexes}
\label{ss:twisted complexes}

We may combine all three envelope operations, obtaining a dg category $\Tw\circ \AB\circ\SB(\CS)$, which we will abbreviate as $\Twix(\CS)$.  Objects of this category are called (finite) twisted complexes; formally they are expressions
\[
\tw_{\boldsymbol{\a}}\Big(\bigoplus_{a\in A} q^{i_a} X_a\Big)
\]
in which $A$ is a finite indexing set, $i_a\in \Gamma$, $X_a\in \CS$, and $\boldsymbol{\a}$ is an $A\times A$-matrix of morphisms $\a_{a,b}\in \Hom_\CS^{\iota-i_a+i_b}(X_b,X_a)$ satisfying the equations
\begin{equation}\label{eq:matrix MC}
(-1)^{i_a} d_\CS(\a_{a,b}) + \sum_{c\in I} \a_{a,c}\circ \a_{c,b} = 0
\end{equation}
for all $a,b\in A$.  We refer to \eqref{eq:matrix MC} as the \emph{matrix Maurer-Cartan equation}.

Let $\XB = \tw_{\boldsymbol{\a}}\Big(\bigoplus_{a\in A} q^{i_a} X_a\Big)$ and $\YB=\tw_{\boldsymbol{\b}}\Big(\bigoplus_{b\in B} q^{j_b} Y_b\Big)$ be twisted complexes.  A degree $k$ morphism $\mathbf{f}\colon \YB\to \XB$ will be given as an $A\times B$-matrix of morphisms $\mathbf{f}=\{f_{a,b}\}_{a,b}$ where $f_{a,b}\in \Hom_{\CS}^{k-i_a+j_b}(Y_b,X_a)$, with differential
\[
d_{\Twix(\CS)}(\{f_{a,b}\}_{a,b}) = \left\{
(-1)^{i_a} d_\CS(f_{a,b})
+\sum_{a'\in A}\a_{a,a'}\circ f_{a',b}
-(-1)^{\ip{\iota,k}}\sum_{b'\in B}f_{a,b'}\circ \b_{b',b}
\right\}_{a,b}.
\]

\begin{remark}\label{rmk:components of morphisms}
We warn the reader that the phrase ``$a,b$-component of $\mathbf{f}$'' is ambiguous, as it may refer to the literal $a,b$-component $\mathbf{f}_{a,b}\in \Hom_{\SB(\CS)}^k(q^{j_b}Y_b,q^{i_a}X_a)$, or it may refer to the \emph{bare} $a,b$-component $f_{a,b}\in \Hom_{\CS}^{k-i_a+j_b}(Y_b,X_a)$. The two are related by
\[
\mathbf{f}_{a,b} = (f_{a,b})^{i_a}_{j_b}
\]
in the notation of \eqref{eq:fij}.  We will always prefer to write $\mathbf{f}$ in terms of its bare components $f_{a,b}$.
\end{remark}

\begin{remark}
Combining what has been said already about shifts, finite direct sums, and twists, we may obtain an ``equational'' characterization of twisted complexes.  For a precise statement, let $A$ be a finite set and suppose we are given an $A$-indexed family of elements $i_a\in \Gamma$ and objects $X_a\in \CS$, as well as a $A\times A$-indexed family of morphisms $\boldsymbol{\a}= \{\a_{a,b}\in \Hom_{\CS}^{\iota-i_a+i_b}(X_b,X_a)\}_{a,b\in A}$ satisfying the matrix Maurer-Cartan equation \eqref{eq:matrix MC}.

Then the object $Y= \tw_{\boldsymbol{\a}}(\bigoplus_a q^{i_a} X_a)$ is uniquely characterized up to isomorphism by existence of $A$-indexed families of maps $\sigma_a\in \Hom_\CS^{i_a}(X_a,Y)$ and $\pi_a\in \Hom_\CS^{-i_a}(Y,X_a)$
such that
\[
d_\CS(\sigma_a) = \sum_{b\in A}  \sigma_b\circ \a_{b,a} \, , \qquad d_\CS(\pi_a)=  -(-1)^{\ip{\iota,i_a}}\sum_{b\in A} \a_{a,b}\circ \pi_b\, ,
\]
\[
\pi_a\circ \sigma_b = \begin{cases} \id_{X_a} & \text{ if $a=b$}\\ 0& \text{ otherwise}\end{cases}\, , \qquad 
\sum_{a\in A} \sigma_a\circ \pi_a = \id_Y
\]
In particular dg functors preserve twisted complexes: if $F\colon \CS\to \DS$ is a dg functor and $Y\cong \tw_{\boldsymbol{\a}}(\bigoplus_a q^{i_a} X_a)$, then $F(Y)\cong \tw_{F(\boldsymbol{\a})}(\bigoplus_a q^{i_a} F(X_a))$.
\end{remark}

A twisted complex is said to be \emph{one-sided} if $I$ admits a partial order with respect to which $\a$ is strictly lower triangular, i.e.~ $\a_{a,b}=0$ unless $a>b$.

\begin{definition}\label{def:pretr}
Let $\pretr(\CS)$, called the \emph{pretriangulated envelope of $\CS$}, denote the full subcategory of $\Twix(\CS)$ whose objects are the finite twisted complexes which are one-sided (with respect to some partial order on the indexing set).
\end{definition}

\begin{definition}\label{def:cone and cocone}
If $\xi\in \Hom^\iota_\CS(X,Y)$ is closed then we have the \emph{extension} of $X$ by $Y$, defined by
\[
(X\buildrel\xi\over \to Y) := \tw_{\smMatrix{0&0\\ \xi&0}}(X\oplus Y).
\]
If $f\in \Hom^0_{\CS}(X,Y)$ is closed and degree zero, then we must shift either $X$ or $Y$ in order to give $f$ degree $\iota$.  The resulting extensions are callled the \emph{cone} and \emph{cocone}, respectively:
\[
\begin{tikzpicture}[anchorbase]
\node at (-2,0) {$\Cone(f)$};
\node at (-1,0) {$:=$};
\node at (-.6,0) {$\Big($};
\node (a) at (0,0) {$q^{-\iota}X$};
\node (b) at (2,0) {$Y$};
\node at (2.3,0) {$\Big)$};
\path[-stealth]
(a) edge node[above] {$f^0_{-\iota}$} (b);
\end{tikzpicture}
\qquad \qquad
\begin{tikzpicture}[anchorbase]
\node at (-2.2,0) {$\mathrm{Cocone}(f)$};
\node at (-1,0) {$:=$};
\node at (-.4,0) {$\Big($};
\node (a) at (0,0) {$X$};
\node (b) at (2,0) {$q^{\iota}Y$};
\node at (2.5,0) {$\Big)$};
\path[-stealth]
(a) edge node[above] {$f^{\iota}_0$} (b);
\end{tikzpicture},
\]
where we are using the notation from \eqref{eq:fij}.
\end{definition}

The obvious analogue of Propositions \ref{prop:suspended envelope}, \ref{prop:additive envelope}, and \ref{prop:twisted envelope} holds for $\pretr$.  Namely, $\pretr$ has the structure of an idempotent monad acting on $\dgcat$, and $\pretr(\CS)\cong \CS$ if and only if $\CS$ is closed under taking finite one-sided twisted complexes.  The monad structure on $\Twix$ and $\pretr$ was observed already in \cite{BonKap90}.

\begin{remark}\label{rmk:negatively graded}
Suppose the grading group $\Gamma$ comes equipped with a subset $\Gamma^-\subset \Gamma$ which is closed under addition, and $0\not\in \Gamma^-$, equivalently $\Gamma^-\cap (-\Gamma^-)=\emptyset$.

The choice of $\Gamma^-$ endows $\Gamma$ with a partial order, by declaring  $i\leq j$ if $i-j\in \Gamma^-\sqcup \{0\}$. As usual we write $i<j$ if $i\leq j$ and $i\neq j$.  Note that we do not require that $\Gamma = \Gamma^-\sqcup \{0\}\sqcup (-\Gamma^-)$, i.e.~ not every element of $\Gamma$ is comparable to 0.

Now, assume that $\iota>0$,  and suppose $\CS$ is a dg category with morphisms supported in degrees $|f|\leq 0$.  Then $\Twix(\CS)$ is equivalent to $\pretr(\CS)$.  

Indeed, if $\XB = \tw_{\boldsymbol{\a}}(\bigoplus_{a\in A} X_a)$ is an object of $\Twix(\CS)$ then $\a_{a,b}\in \Hom^{\iota-i_a+i_b}_\CS(X_b,X_a)$ is zero unless $\iota-i_a+i_b\leq 0$, i.e.~$i_a\geq i_b+\iota$.  In particular $\boldsymbol{\a}$ is strictly lower-triangular with respect to the partial order given by declaring $a\geq b$ if $i_a\geq i_b$.
\end{remark}

\begin{example}\label{ex:pretriangulated hull of additive cat}
In this example we work with the standard grading group for complexes, with the cohomological convention for differentials, i.e.~ take $\Gamma=\Z$ with $\iota =1$ and $\ip{i,j}=ij$ (mod 2).

Let $\CS$ be an ordinary $\k$-linear category.  We regard $\CS$ as a dg category with trivial grading (all morphisms have degree zero) and differential (all morphisms are closed).  Suppose further that $\AB(\CS)\cong \CS$, i.e.~ $\CS$ is additive.   Then $\pretr(\CS)$ is equivalent to the usual (dg) category $\Ch^b(\CS)$ of finite complexes over $\CS$.

Indeed, if $\XB = \tw_{\boldsymbol{\a}}(\bigoplus_{a\in A} q^{i_a}X_a)$ is an object of $\pretr(\CS)$, then the components of the twist $\boldsymbol{\a}$ are $\a_{a,b}\in \Hom^{1-i_a+j_b}_\CS(X_b,X_a)$, which is zero unless $j_b-i_a=1$. Thus, we may rearrange
\[
\tw_\a\Big(\bigoplus_{a\in A} q^{i_a} X_a\Big) \cong \tw_\a \Big(\bigoplus_{i\in \Z} q^i \bigoplus_{\substack{a\in A \\ i_a = i}} X_a\Big) =: \tw_{\boldsymbol{\d}} \Big(\bigoplus_{i\in \Z} q^i X^i\Big)
\]
Which may be expressed as a complex in the traditional sense
\begin{equation}\label{eq:CX ex}
\begin{tikzpicture}[baseline=0cm]
\node (a) at (-.5,0)  {$\cdots$};
\node (b) at (1.5,0) {$X^i$};
\node (c) at (3,0) {$X^{i+1}$};
\node (d) at (5,0)  {$\cdots$};
\path[-stealth]
(a) edge node[above] {$\d_{i,i-1}$} (b)
(b) edge node[above] {$\d_{i+1,i}$} (c)
(c) edge node[above] {$\d_{i+2,i+1}$} (d);
\end{tikzpicture}
\end{equation}
Note that $d_\CS=0$, so the Maurer-Cartan equation \eqref{eq:matrix MC} for $\boldsymbol{\d}$ is the usual equation $\boldsymbol{\d}^2=0$ for differentials.
\end{example}

\subsection{(Co)completions}
\label{ss:cocompletion}
We conclude this section with a discussion on the process of adjoining infinite one-sided twisted complexes.  There are two kinds of such infinite twisted complexes and, importantly, adjoining them to a dg category is not an idempotent operation.

\begin{definition}\label{def:ch amalg}
Let $\CS$ be a dg category.  Define $\pretr^{\amalg}(\CS)$ to be the category whose objects are formal expressions of the form
\[
\tw_\a\left( \coprod_{a\in A} q^{i_a} X_a\right)
\]
where $A$ is a partially ordered set satisfying the property that for each $a\in A$ the set $\{b\in A\:|\: b\geq a\}$ is finite, and $\a$ is an $A\times A$-matrix of morphisms $\a_{a,b}\in \Hom_\CS^{\iota+i_a-i_b}(X_b,X_a)$ satisfying
\begin{enumerate}
\item $\a_{a,b}=0$ unless $a>b$.
\item the matrix Maurer-Cartan equation \eqref{eq:matrix MC} is satisfied for all $a,b\in A$.
\end{enumerate}
The hom complexes in $\pretr^{\amalg}(\CS)$ are defined by
\[
\Hom_{\pretr^{\amalg}(\CS)}^l\left(\tw_\b\left( \coprod_{b\in B} q^{j_b} Y_b\right),\tw_\a\left( \coprod_{a\in A} q^{i_a} X_a\right)\right):=\prod_{b\in B}\coprod_{a\in A} \Hom_{\CS}^{l+j_b-i_a}(Y_b,X_a).
\]
where the product and coproduct are taken in the category of $\Gamma$-graded $\k$-modules.
An element of the above hom complex is a strictly lower triangular, column-finite $A\times B$-matrix $f=(f_{a,b})_{(a,b)\in A\times B}$ (in which $f_{a,b}\in \Hom_\CS^{l+j_b-i_a}(Y_b,X_a)$).  The column-finiteness means that for each $b\in B$ we have $f_{a,b}=0$ for all but finitely many $a\in A$, as indicated by the coproduct over $a\in A$.  In the language of matrices, the differential of $f=(f_{a,b})$ of degree $l=|f|$ is given by
\[
d(f)_{a,b} = (-1)^{i_a} d(f_{a,b}) + \sum_{a'\in A} \a_{a,a'}\circ f_{a',b} - (-1)^{\ip{\iota,l}} \sum_{b'\in B} f_{a,b'}\circ \b_{b',b}.
\]
\end{definition}

Replacing column-finiteness with row-finiteness defines $\pretr^\Pi$.

\begin{definition}\label{def:ch Pi}
Let $\CS$ be a dg category.  Define $\pretr^{\Pi}(\CS)$ to be the category whose objects are formal expressions of the form
\[
\tw_\a\left( \prod_{a\in A} q^{i_a} X_a\right)
\]
where $A$ is a partially ordered set satisfying the property that for each $a\in A$ the set $\{b\in A\:|\: b\leq a\}$ is finite, and $\a$ is an $A\times A$-matrix of morphisms $\a_{a,b}\in \Hom_\CS^{\iota+i_a-i_b}(X_b,X_a)$ satisfying
\begin{enumerate}
\item $\a_{a,b}=0$ unless $a>b$.
\item the matrix Maurer-Cartan equation \eqref{eq:matrix MC} is satisfied for all $a,b\in A$.
\end{enumerate}
The hom complexes in $\pretr^{\Pi}(\CS)$ are defined by
\[
\Hom_{\pretr^{\amalg}(\CS)}^l\left(\tw_\b\left( \prod_{b\in B} q^{j_b} Y_b\right),\tw_\a\left( \prod_{a\in A} q^{i_a} X_a\right)\right):=\prod_{a\in A}\coprod_{b\in B} \Hom_{\CS}^{l+j_b-i_a}(Y_b,X_a).
\]
where the product and coproduct are taken in the category of $\Gamma$-graded $\k$-modules.
An element of the above hom complex is a strictly lower triangular, row-finite $A\times B$-matrix $f=(f_{a,b})_{(a,b)\in A\times B}$ (in which $f_{a,b}\in \Hom_\CS^{l+j_b-i_a}(Y_b,X_a)$).  The row-finiteness means that for each $a\in A$ we have $f_{a,b}=0$ for all but finitely many $b\in B$, as indicated by the coproduct over $b\in B$.
\end{definition}

\section{Opposites and tensor products of dg categories, and their envelopes}

\subsection{Opposite category}
\label{ss:opp}
\begin{definition}\label{def:opposite cat}
If $\CS$ is a dg category we let $\CS^{\op}$ denote the \emph{opposite dg category}.  Objects of $\CS^\op$ are formal expressions $X^{\op}$ where $X\in \CS$, with hom complexes
\[
\Hom_{\CS^{\op}}(Y^{\op},X^{\op}) :=\Hom_\CS(X,Y).
\]
For each $f\in \Hom_{\CS}(X,Y)$ we let $f^{\op}$ denote the corresponding element of $\Hom_{\CS^\op}(Y^{\op},X^{\op})$.  The grading, differential, and composition in $\CS^{\op}$ are defined by
\[
|f^{\op}| = |f| \ ,\qquad d_{\CS^{\op}}(f^{\op}) = d(f)^{\op} \ , \qquad f^{\op}\circ g^{\op}:= (-1)^{\ip{|f|,|g|}}(g\circ f)^{\op}.
\]
\end{definition}
\begin{remark}\label{rmk:contravariant functor}
If $F\colon \CS^{\op}\to \DS$ is a dg functor, then we sometimes we drop the superscript $(-)^\op$, regarding $F$ as a \emph{contravariant dg functor} $\CS\to \DS$.
\end{remark}

\begin{proposition}\label{prop:envelopes and op}
If $\EB$ is one of the envelope operations $\EB\in \{\SB,\AB,\Tw,\Twix,\pretr\}$, then for any dg category $\CS$ we have $\EB(\CS)^{\op}\cong \EB(\CS^{\op})$.
\end{proposition}
\begin{proof}
We define $\SB(\CS)^{\op}\buildrel\cong\over\rightarrow \SB(\CS^{\op})$ by
\[
\displaystyle
(q^k X)^{\op}\mapsto q^{-k}X^{\op} \ ,\qquad (f^i_j)^{\op} \mapsto (-1)^{\ip{i+j,i+|f|}}(f^{\op})^{-j}_{-i} 
\]
We define $\AB(\CS)^{\op}\buildrel\cong\over\rightarrow \AB(\CS^{\op})$ by
\[
(\bigoplus_{a\in A} X_a)^{\op}\mapsto \bigoplus_{a\in A} X_a^\op \ ,\qquad \{f_{a,b}\}_{a,b}^{\op}\mapsto \{f_{a,b}^\op\}_{b,a}
\]
We define $\Tw(\CS)^{\op}\buildrel\cong\over\rightarrow \Tw(\CS^{\op})$  by
\[
\tw_\a(X)^{\op}\leftrightarrow \tw_{-\a^{\op}}(X^{\op}) \ , \qquad f\mapsto f.
\]
It is an exercise to check that each of these defines a dg functor (each is clearly invertible).  Combining all three isomorphisms above defines $\Twix(\CS)^{\op}\buildrel\cong \over \rightarrow \Twix(\CS^{\op})$, which acts on objects by
\[
\tw_{\boldsymbol{\a}} \Big(\bigoplus_{a\in A} q^{i_a} X_a\Big)^{\op} \mapsto \tw_{\boldsymbol{\a}'} \Big(\bigoplus_{a\in A} q^{-i_a} X_a^{\op}\Big) 
\]
where $\boldsymbol{\a}'$ has components
\begin{equation}\label{eq:opposite sign a}
\a'_{b,a}  = (-1)^{1+\ip{i_a+i_b,\iota}+i_b} \a_{a,b}
\end{equation}
A degree $l$ morphism $\mathbf{f}^{\op}$, in which $\mathbf{f}$ has bare components $f_{a,b}\in \Hom_\CS^{l-i_a+j_b}(Y_b,X_a)$, maps as
\begin{equation}\label{eq:opposite sign f}
\mathbf{f}^{\op} = \Big(\{f_{a,b}\}_{a,b}\Big)^\op \mapsto \Big\{(-1)^{\ip{i_a+j_b,l+j_b}} (f_{a,b})^{\op}\Big\}_{b,a}
\end{equation}
\end{proof}

\begin{remark}
If $F\colon \CS \to \DS$ is a contravariant functor, then there is an induced contravariant functor $\EB(F)\colon \EB(\CS)\to \EB(\DS)$ obtained as the composition
\[
\EB(\CS)^{\op}\to \EB(\CS^{\op})\to \EB(\DS).
\]
where the first functor is from Proposition \ref{prop:envelopes and op}, and the second is the extension of the dg functor $F$ to the envelope.
\end{remark}

\begin{example}
If $\CS$ is a $\k$-linear category, thought of as a dg category with trivial grading and differential, then the sign rule \eqref{eq:opposite sign f} simplifies to 
\[
\mathbf{f}^{\op} = \Big(\{f_{a,b}\}_{a,b}\Big)^\op \mapsto \Big\{(-1)^{\ip{l,i_a}} (f_{a,b})^{\op}\Big\}_{b,a}
\]
since $f_{a,b}\neq 0$ implies $l=|\mathbf{f}|=j_b-i_a$.
\end{example}

\begin{example}
In this example let $\Gamma=\Z$, $\iota=1$ and $\ip{i,j}=ij$ (mod 2).  If $\CS$ and $\DS$ are $\k$-linear categories and $F\colon \CS\to \DS$ is a contravariant functor then the induced contravariant functor $\Ch^b(\CS)\to\Ch^b(\DS)$ sends
\[
\begin{tikzpicture}[baseline=0cm]
\node (a) at (0,0)  {$\cdots$};
\node (b) at (1.5,0) {$X^i$};
\node (c) at (3,0) {$X^{i+1}$};
\node (d) at (4.5,0)  {$\cdots$};
\path[-stealth]
(a) edge node[above] {} (b)
(b) edge node[above] {\tiny{$\d_{i+1,i}$}} (c)
(c) edge node[above] {} (d);
\end{tikzpicture}
\qquad\mapsto \qquad 
\begin{tikzpicture}[baseline=0cm]
\node (a) at (-.5,0)  {$\cdots$};
\node (b) at (1.5,0) {$F(X^{i+1})$};
\node (c) at (5,0) {$F(X^{i})$};
\node (d) at (6.8,0)  {$\cdots$};
\path[-stealth]
(a) edge node[above] {} (b)
(b) edge node[above] {\tiny $(-1)^i F(\d_{i+1,i})$} (c)
(c) edge node[above] {} (d);
\end{tikzpicture}
\]
and a degree $l$-morphism of complexes (not necessarily closed) maps to
\[
\begin{tikzpicture}[baseline=-1cm]
\node (a) at (0,0)  {$\cdots$};
\node (b) at (1.5,0) {$X^i$};
\node (c) at (4,0) {$X^{i+1}$};
\node (d) at (5.5,0)  {$\cdots$};
\node (a1) at (0,-2)  {$\cdots$};
\node (b1) at (1.5,-2) {$Y^{j}$};
\node (c1) at (4,-2) {$Y^{j+1}$};
\node (d1) at (5.5,-2)  {$\cdots$};
\path[-stealth]
(a) edge node[above] {} (b)
(b) edge node[above] {\tiny{$\d_{i+1,i}$}} (c)
(c) edge node[above] {} (d)
(a1) edge node[above] {} (b1)
(b1) edge node[above] {\tiny{$\epsilon_{j+1,j}$}} (c1)
(c1) edge node[above] {} (d1)
(b1) edge node[right] {\tiny{$f_{i,j}$}} (b)
(c1) edge node[right] {\tiny{$f_{i+1,j+1}$}} (c);
\end{tikzpicture}
\mapsto
\begin{tikzpicture}[baseline=-1cm]
\node (a) at (-.5,0)  {$\cdots$};
\node (b) at (1.5,0) {{$F(X^{i+1})$}};
\node (c) at (5,0) {{$F(X^{i})$}};
\node (d) at (6.8,0)  {$\cdots$};
\node (a1) at (-.5,-2)  {$\cdots$};
\node (b1) at (1.5,-2) {{$F(Y^{j+1})$}};
\node (c1) at (5,-2) {{$F(Y^{j})$}};
\node (d1) at (6.8,-2)  {$\cdots$};
\path[-stealth]
(a) edge node[above] {} (b)
(b) edge node[above] {\tiny{$(-1)^i F(\d_{i+1,i})$}} (c)
(c) edge node[above] {} (d)
(a1) edge node[above] {} (b1)
(b1) edge node[above] {\tiny{$(-1)^j F(\epsilon_{j+1,j})$}} (c1)
(c1) edge node[above] {} (d1)
(c) edge node[right] {\tiny{$(-1)^{li}F(f_{i,j})$}} (c1)
(b) edge node[right] {\tiny{$(-1)^{li+l}F(f_{i+1,j+1})$}} (b1);
\end{tikzpicture}
\]
where $i-j = l$.  That this is a well-defined contravariant dg functor $\Ch^b(\CS)\to \Ch^b(\DS)$ can be checked directly, or can be deduced from Proposition \ref{prop:envelopes and op}.

\end{example}
\subsection{Tensor product of dg categories}
\label{ss:tensor}

\begin{definition}\label{def:tensor prod}
If $\CS$ and $\DS$ are dg categories, then we let $\CS\otimes \DS$ be the dg category with $\Obj(\CS\otimes \DS) = \Obj(\CS)\times \Obj(\DS)$ and morphism complexes
\[
\Hom_{\CS\otimes \DS}\Big((X,Y),(X',Y')\Big) :=\Hom_\CS(X,X')\otimes\Hom_\DS(Y,Y'),
\]
with composition $(f'\otimes g')\circ (f\otimes g)=(-1)^{\ip{|g'|,|f|}} (f'\circ f)\otimes (g'\circ g)$.
\end{definition}
\begin{remark}
In order to avoid a proliferation of parentheses, we will often write $X\otimes Y$ to denote the object $(X,Y)$ in $\CS\otimes \DS$.
\end{remark}

\begin{remark}
If $f$ and $g$ are invertible morphisms in $\CS$ and $\DS$, respectively, then $f\otimes g$ is invertible in $\CS\otimes \DS$, with inverse
\begin{equation}\label{eq:inverse of tensor}
(f\otimes g)\inv = (-1)^{\ip{|f|,|g|}}(f\inv\otimes g\inv).
\end{equation}
\end{remark}

Next, for each envelope construction $\EB\in \{\SB,\AB,\Tw,\Twix,\pretr\}$ and each pair of dg categories $\CS,\DS$ we wish to define a fully faithful functor $\EB(\CS)\otimes \EB(\DS)\hookrightarrow \EB(\CS\otimes \DS)$.  We denote this functor as a binary operation $\ast_\EB$ (or simply $\ast$ when the envelope $\EB$ is understood).

\begin{definition}\label{def:S and tensor}
Let $\CS,\DS$ be dg categories.  Define $\ast_\SB \colon \SB(\CS)\otimes \SB(\DS)\to \SB(\CS\otimes \DS)$ on objects by
\[
q^iX \ast_\SB q^jY := q^{i+j} X\otimes Y,
\]
and on morphisms by 
\begin{equation}\label{eq:tensor sign}
f^{i'}_i\ast_\SB g^{j'}_j := (-1)^{\ip{k,j'}+\ip{i,j'+l+j}}(f\otimes g)^{i'+j'}_{i+j},
\end{equation}
where $f\in \Hom_{\CS}^k(X, X')$, $g\in \Hom_{\DS}^l(Y, Y')$, and $i,i'j,j'\in \Gamma$, and we are using the notation from \eqref{eq:fij}.
\end{definition}


\begin{remark}\label{rmk:tensor sign}
The sign in \eqref{eq:tensor sign} can be justified (and remembered) by considering the Koszul sign rule involved in the following rearrangement of factors:
\[
(\phi_{X',i'}\circ f\circ \phi_{X,i}\inv) \otimes (\phi_{Y',j'}\circ g\circ \phi_{Y,j}\inv) = \pm (\phi_{X',i'}\otimes \phi_{Y',j'}) \circ (f\otimes g) \circ (\phi_{X,i}\otimes \phi_{Y,j})\inv
\]
In particular, $(-1)^{\ip{k+i,j'}}$ comes from commuting $\phi_{Y',j'}$ past $\phi_{X,i}\inv$ and $f$, the sign $(-1)^{\ip{i,l}}$ comes from commuting $g$ past $\phi_{X,i}\inv$, and the contribution of $(-1)^{\ip{i,j}}$ comes from the fact that
\[
(\phi_{X,i}\inv\otimes \phi_{Y,j}\inv) = (-1)^{\ip{i,j}} (\phi_{X,i}\otimes \phi_{Y,j})\inv,
\]
using \eqref{eq:inverse of tensor}.
\end{remark}

For additive envelopes we have the following.
\begin{definition}\label{def:A and tensor}
Let $\CS,\DS$ be dg categories.  Define $\ast_\AB\colon \AB(\CS)\otimes \AB(\DS)\to \AB(\CS\otimes \DS)$ on objects by
\[
\left(\bigoplus_{a\in A}X_a\right)\ast_\AB \left(\bigoplus_{b\in B}Y_b\right) :=\bigoplus_{(a,b)\in A\times B} X_a\otimes Y_b,
\]
and on morphisms by
\[
(f\ast g)_{(a,b),(a',b')} = f_{a,a'}\otimes g_{b,b'},
\]
where $f\colon \XB'\to \XB$ and $g\colon \YB'\to\YB$ are morphisms with components denoted by $f_{a,a'}$ and $g_{b,b'}$ respectively, with $a,a',b,b'$ running over appropriate finite sets of indices.
\end{definition}

And finally, for twisted envelopes we have the following.

\begin{definition}\label{def:Tw and tensor}
Let $\CS,\DS$ be dg categories.  Define $\ast_{\Tw}\colon \Tw(\CS)\otimes \Tw(\DS)\to \Tw(\CS\otimes \DS)$ on objects by
\[
\tw_\a(X)\ast_{\Tw} \tw_\b(Y) :=\tw_{\a\otimes \id+\id\otimes \b}(X\otimes Y)
\]
and on morphisms by
\[
f\ast_{\Tw} g := f\otimes g.
\]
\end{definition}

\begin{proposition}\label{prop:envelope and tensor}
Let $\EB$ denote one of the envelope construction $\EB\in \{\SB,\AB,\Tw\}$. The operation $\ast_{\EB}$ from Definitions \ref{def:S and tensor} (resp.~Definition \ref{def:A and tensor}, resp.~Definition \ref{def:Tw and tensor}) is a fully faithful dg functor $\EB(\CS)\otimes \EB(\DS)\hookrightarrow \EB(\CS\otimes \DS)$. 
\end{proposition}
\begin{proof}
Routine.
\end{proof}


\begin{remark}\label{rmk:Twix and tensor}
Combining all of the above defines $\ast_{\Twix}\colon \Twix(\CS)\otimes \Twix(\DS)\to \Twix(\CS\otimes \DS)$.   In the interest of thoroughness, let us provide explicit formulas.  Suppose we have finite twisted complexes $\XB=\tw_{\boldsymbol{\a}}(\bigoplus_{a\in A} q^{i_a} X_a)$ and $\YB=\tw_{\boldsymbol{\b}}(\bigoplus_{b\in B}q^{j_b} Y_b)$. Then
\[
\XB\ast_{\Twix(\CS)} \YB := \tw_{\a\otimes \id+\id\otimes \b} \Big(\bigoplus_{a,b} q^{i_a+j_b} X_a\otimes Y_b\Big).
\]
On morphisms: if $\mathbf{f}$ is a degree $\mathbf{k}$ morphism from $\tw_{\boldsymbol{\a}}(\bigoplus_{a\in A} q^{i_a} X_a)$ to $\tw_{\boldsymbol{\a}'}(\bigoplus_{a'\in A'} q^{i'_{a'}} X_{a'})$ and $\mathbf{g}$ is a degree $\mathbf{l}$ morphism from $\tw_{\boldsymbol{\b}}(\bigoplus_{b\in B}q^{j_b} Y_b)$ to $\tw_{\boldsymbol{\b}'}(\bigoplus_{b'\in B'}q^{j'_{b'}} Y'_{b'})$ then $\mathbf{f}\otimes \mathbf{g}$ is an $(A'\times B')\times (A\times B)$ matrix of morphisms:
\[
\mathbf{f}\ast \mathbf{g} := \Big\{(-1)^{\ip{\mathbf{k}-i'_{a'}+i_a,j'_{b'}}+\ip{i_a,\mathbf{l}}} f_{a',a}\otimes g_{b',b} \Big\}_{(a',b'),(a,b)}
\]
where we are denoting all morphisms via their bare components $f_{a',a}\in \Hom_{\CS}^{k-i'_{a'}+i_a}(X_a,X_{a'})$ and $g_{b',b}\in \Hom_{\DS}^{l-j'_{b'}+j_b}(Y_b, Y_{b'})$, as in Remark \ref{rmk:components of morphisms}.  In particular, the Maurer-Cartan endomorphism defining $\XB \ast_{\Twix}\YB$ is given by
\[
\boldsymbol{\a}\ast \id + \id\ast \boldsymbol{\b} = \Big\{
 (-1)^{\ip{\iota-i_{a'}+i_a,j_{b'}}} \a_{a',a}\otimes \id_{b',b}
 +(-1)^{\ip{\iota-i_{a'}+i_a,j_{b'}}+\ip{i_a,\iota}} \id_{a',a}\otimes \b_{b',b}
 \Big\}_{(a',b'),(a,b)}
\]
Of course, the components of the identity maps $\id_{a',a}$ and $\id_{b',b}$ are zero unless $a'=a$, respectively $b'=b$.
\end{remark}

\begin{remark}\label{rmk:pretr and tensor}
Continuing Remark \ref{rmk:Twix and tensor}, if $\boldsymbol{\a}$ and $\boldsymbol{\b}$ are strictly lower triangular with respect to some partial order on $A$ and $B$, respectively, then it is clear that $\boldsymbol{\a}\ast \id+\id\ast \boldsymbol{\b}$ is strictly lower triangular with respect to the product partial order $(a',b')\geq (a,b)$  iff $a'\geq a$ and $b'\geq b$, hence $\ast_{\Twix(\CS)}$ restricts to a fully faithful dg functor $\ast_{\pretr}\colon \pretr(\CS)\otimes \pretr(\DS)\to \pretr(\CS\otimes \DS)$.
\end{remark}

\begin{example}\label{ex:tensor of complexes}
Let us continue the Example \ref{ex:pretriangulated hull of additive cat} on the category of complexes over an additive category $\CS$. Set $\Gamma=\Z$ with $\iota =1$ and $\ip{i,j}=ij$ (mod 2), and consider the category of complexes $\Ch^b(\CS)\cong \pretr(\CS)$.  We may express a complex over $\CS$ in our notation as $\tw_{\boldsymbol{\d}}(\bigoplus_{i\in \Z} q^i X^i)$, or in the more traditional way as in \eqref{eq:CX ex}.

For additive categories $\CS,\DS$, the functor $\Ch^b(\CS)\otimes \Ch^b(\DS)\to \Ch^b(\CS\otimes \DS)$ sends a pair of complexes
\[
\XB = \tw_{\boldsymbol{\d}}\Big(\bigoplus_{i\in \Z} q^i X^i \Big)=
\begin{tikzpicture}[baseline=0cm]
\node (a) at (-.5,0)  {$\cdots$};
\node (b) at (1.5,0) {$X^i$};
\node (c) at (3,0) {$X^{i+1}$};
\node (d) at (5,0)  {$\cdots$};
\path[-stealth]
(a) edge node[above] {$\d_{i,i-1}$} (b)
(b) edge node[above] {$\d_{i+1,i}$} (c)
(c) edge node[above] {$\d_{i+2,i+1}$} (d);
\end{tikzpicture}
\]
\[
\YB = \tw_{\boldsymbol{\epsilon}}\Big(\bigoplus_{j\in \Z} q^j Y_j\Big) =
\begin{tikzpicture}[baseline=0cm]
\node (a) at (-.5,0)  {$\cdots$};
\node (b) at (1.5,0) {$Y^j$};
\node (c) at (3,0) {$Y^{j+1}$};
\node (d) at (5,0)  {$\cdots$};
\path[-stealth]
(a) edge node[above] {$\epsilon_{j,j-1}$} (b)
(b) edge node[above] {$\epsilon_{j+1,j}$} (c)
(c) edge node[above] {$\epsilon_{j+2,j+1}$} (d);
\end{tikzpicture}
\]
to $\XB\ast \YB = \tw_{\boldsymbol{\d}\ast \id+\id\ast \boldsymbol{\epsilon}}\Big(\bigoplus_{i,j\in \Z} q^{i+j} X^i\otimes Y^j\Big)$.
The sign rule for the components of $f\ast g$ reduces to
\[
\mathbf{f}\ast \mathbf{g} = \Big\{(-1)^{\ip{i,l}} f_{i',i}\otimes g_{j',j}\Big\}_{(i',j'),(i,j)}
\]
where $i'-i=k$ and $j'-j=l$.  So in particular $\boldsymbol{\d}\ast \id+\id\ast \boldsymbol{\epsilon}$ is a sum of terms $\d_{i+1,i}\otimes \id_{j,j}$ and $(-1)^i \id_{i,i} \otimes \epsilon_{j+1,j}$.  In other words, we can identify $\XB\ast \YB$ as the total complex of a bicomplex:
\[
\XB\ast \YB = \Tot\left(
\begin{tikzpicture}[baseline=1cm]
\node (b0) at (0,-2) {$\vdots$};
\node (c0) at (4,-2) {$\vdots$};
\node (a1) at (-2,0) {$\cdots$};
\node (b1) at (0,0) {$X^i\otimes Y^j$};
\node (c1) at (4,0) {$X^{i+1}\otimes Y^j$};
\node (d1) at (6,0) {$\cdots$};
\node (a2) at (-2,2) {$\cdots$};
\node (b2) at (0,2) {$X^i\otimes Y^{j+1}$};
\node (c2) at (4,2) {$X^{i+1}\otimes Y^{j+1}$};
\node (d2) at (6,2) {$\cdots$};
\node (b3) at (0,4) {$\vdots$};
\node (c3) at (4,4) {$\vdots$};
\path[-stealth]
(a1) edge node[above] {} (b1)
(b1) edge node[above] {$\d_{i+1,i}\otimes \id_{j,j}$} (c1)
(c1) edge node[above] {} (d1)
(a2) edge node[above] {} (b2)
(b2) edge node[above] {} (c2)
(c2) edge node[above] {} (d2)
(b0) edge node[left] {} (b1)
(b1) edge node[left] {$\id_{i,i}\otimes \epsilon_{j+1,j}$} (b2)
(b2) edge node[left] {} (b3)
(c0) edge node[right] {} (c1)
(c1) edge node[right] {} (c2)
(c2) edge node[right] {} (c3);
\end{tikzpicture}
\right)
\]
in which we place the sign $(-1)^i$ on the differentials occuring in the $i$-th column.
\end{example}

\begin{remark}\label{rmk:envelope of multilinear functor}
If we have dg categories $\CS_1,\ldots,\CS_r$ and $\DS$, then iterated applications of Proposition \ref{prop:envelope and tensor} tells us how to lift a multilinear functor $\CS_1\otimes\cdots \otimes \CS_r\to \DS$ to a multilinear functor on the envelopes $\EB(\CS_1)\otimes \cdots \otimes \EB(\CS_r)\to \EB(\DS)$.  Combining with Proposition \ref{prop:envelopes and op} allows us to consider multilinear functors which are contravariant in some arguments as well.

If $\CS_1,\ldots,\CS_r,\DS$ are additive categories, then in view of Example \ref{ex:tensor of complexes}, the extension of multilinear functors to categories of complexes $\Ch^b(\CS_1)\otimes\cdots\otimes \Ch^b(\CS_r)\to \Ch^b(\DS)$ recovers, for example, Bar-Natan's extension of his canopolis formalism \cite{B-N05} to complexes.
\end{remark}

\section{Bimodules and the bar complex}
\label{s:bimodules}

\subsection{Bras and kets}
\label{ss:bras and kets}

\begin{definition}\label{def:bimod}
If $\CS,\DS$ are dg categories, we let $\Bim_{\CS,\DS}$ denote the dg category of $(\CS,\DS)$-bimodules.  An object $M\in\Bim_{\CS,\DS}$ is a collection $\braCket{X}{M}{Y}\in \dgMod{\Gamma}{\k}$ indexed by objects $X\in \Obj(\CS)$, $Y\in \Obj(\DS)$, together with action maps
\begin{equation}\label{eq:bimod}
\braCket{X'}{\CS}{X}\otimes \braCket{X}{M}{Y}\otimes \braCket{Y}{\DS}{Y'}\rightarrow \braCket{X'}{M}{Y'}.
\end{equation}
for all $X,X'\in \Obj(\CS)$ and all $Y,Y'\in \Obj(\DS)$, satisfying the usual associativity and unit axioms.

The \emph{identity bimodule} $\one_\CS\in \Bim_{\CS,\CS}$ (frequently also denoted simply by $\CS$) was introduced implicitly already.  It is defined by  is defined by \eqref{eq:XCY}, with bimodule structure defined by composition of morphisms in $\CS$.
\end{definition}

Now we discuss the notion of tensor product of bimodules $M\in \Bim_{\AS,\BS}$ and $N\in \Bim_{\BS,\CS}$.  This is commonly denoted $M\otimes_\BS N\in \Bim_{\AS,\CS}$.   When the dg category $\BS$ is understood, we will sometimes omit it from the notation, writing $M\star N$, or $M\cdot N$, or $MN$ for the tensor product $M\otimes_\BS N$.  The tensor product is defined by
\[
\braCket{X}{M N}{Z} := \left(\bigoplus_{Y\in \Obj(\BS)} \braCket{X}{M}{Y}\otimes \braCket{Y}{N}{Z}\right)\bigg/ (mf\otimes n \sim m\otimes fn).
\]
for all $X\in \Obj(\AS)$, $Z\in \Obj(\CS)$.

\begin{remark}
Note that we are using the naive, underived tensor product over $\BS$ here.  The derived tensor product will make use of the bar complex of $\BS$, introduced later.
\end{remark}

Some texts define a $(\CS,\DS)$-bimodule to be a functor $\CS\otimes \DS^{\op}\to \dgMod{\Gamma}{\k}$.  The following describes the equivalence between these notions.
\begin{lemma}
Let $F\colon \CS\otimes \DS^{\op} \to \dgMod{\Gamma}{\k}$ be a dg functor.  Define a bimodule $B\in \Bim_{\CS,\DS}$ by
\[
\braCket{X}{B}{Y}:=F(X,Y^{\op}),
\]
with action
\[
\begin{tikzpicture}
\node (a) at (0,0) {$\braCket{X'}{\CS}{X}\otimes \braCket{X}{B}{Y}\otimes \braCket{Y}{\DS}{Y'}$};
\node (b) at (5,0) {$\braCket{X'}{B}{Y'}$};
\node (c) at (0,-.5) {$f\otimes b\otimes g$};
\node (d) at (5,-.5) {$(-1)^{\ip{|b|,|g|}}F(f\otimes g^{\op})(b)$};
\path[-stealth]
(a) edge node[above] {} (b);
\path[|-stealth]
(c) edge node {} (d);
\end{tikzpicture}
\]
This construction defines an isomorphism of dg categories $\Fun(\CS\otimes\DS^{\op},\k)$ and $\Bim_{\CS,\DS}$.
\end{lemma}

\subsection{Yoneda modules}
\label{ss:yoneda}

We regard $\k$ as a dg category with one object $\ast$ and $\End(\ast)=\k$. A \emph{right (resp.~ left) $\CS$-module} is an object of $\Bim_{\k,\CS}$ (resp.~ $\Bim_{\CS,\k}$).  

We will sometimes denote a right $\CS$-module $M$ by the notation $\bra{M}$ with $\braket{M}{X}:=\braCket{\ast}{M}{X}$ for $X\in \Obj(\CS)$ (and similarly for left $\CS$- modules
%

\begin{definition}\label{def:Yoneda}
For each $X\in \Obj(\CS)$ we have the right and left $\CS$-modules $\bra{X}$ and $\ket{X}$ (called the \emph{Yoneda modules} for $X$).
\end{definition}

\begin{lemma}\label{lemma:yoneda embeddings}
The assignment $X\mapsto \bra{X}$ extends to a dg functor $\RS_\CS\colon\CS\rightarrow\Bim_{\k,\CS}$.  Dually, the assignment  $X\mapsto \ket{X}$ extends to a contravariant dg functor $\LS_\CS\colon \CS \to \Bim_{\CS,\k}$.
\end{lemma}
\begin{proof}
The action of $\RS_\CS$ on morphisms is given as follows.  If $f\colon X\to X'$ in $\CS$ then the module map $\RS_\CS(f)\colon \bra{X}\to \bra{X'}$ is defined by
\[
\begin{tikzpicture}
\node (a) at (0,0) {$\braket{X}{Y}$};
\node (b) at (3,0) {$\braket{X'}{Y}$};
\node (c) at (0,-.5) {$m$};
\node (d) at (3,-.5) {$f\circ m$};
\path[-stealth]
(a) edge node[above] {$\RS_\CS(f)$} (b);
\path[|-stealth]
(c) edge node {} (d);
\end{tikzpicture}
\]

Now, let $\LS_\CS\colon \CS \rightarrow \Bim_{\CS,\k}$ denote the contravariant dg functor defined by $\LS_{\CS}(X):= \ket{X}$, with the action on morphisms given as follows.  If $f\colon X\to X'$ in $\CS$ then $\LS_\CS(f)\colon \ket{X'}\to \ket{X}$ is defined by 
\[
\begin{tikzpicture}
\node (a) at (0,0) {$\braket{Y}{X'}$};
\node (b) at (3,0) {$\braket{Y}{X}$};
\node (c) at (0,-.5) {$m$};
\node (d) at (3,-.5) {$(-1)^{\ip{|m|,|f|}} m\circ f$};
\path[-stealth]
(a) edge node[above] {$\LS_\CS(f)$} (b);
\path[|-stealth]
(c) edge node {} (d);
\end{tikzpicture}
\]
We leave it to the reader to verify that $\RS_\CS$ and $\LS_{\CS}$ are dg functors (with $\LS_{\CS}$ contravariant).
\end{proof}

\begin{corollary}
If $Y \cong \tw_\a(\bigoplus_{a\in A} q^{i_a} X_a)$ in $\CS$ then the Yoneda modules $\bra{Y}$ and $\ket{Y}$ satisfy
\[
\bra{Y}\cong \tw_{\bra{\a}}\left(\bigoplus_{a\in A} q^{i_a} \bra{X_a}\right) \, ,  \qquad \quad \ket{Y}\cong \tw_{-\ket{\a}}\left(\bigoplus_{a\in A} q^{-i_a}\ket{X_a}\right).
\]
\end{corollary}

\begin{proof}
The statement follows from the fact that $\bar{-}$ (resp.~$\ket{-}$) is a dg functor (resp.~contravariant dg functor), hence commutes with sums, shifts, and twists.
\end{proof}

The following is straightforward.
\begin{lemma}\label{lemma:yoneda}
Given $B\in \Bim_{\CS,\DS}$ and $X\in \Obj(\CS)$, $Y\in \Obj(\DS)$, we have
\[
\bra{X}\otimes_\CS B\otimes_\DS \ket{Y} \cong \braCket{X}{B}{Y} \cong \Hom_{\Bim_{\CS,\DS}}\Big(\ket{X}\otimes_\k \bra{Y} , B\Big)
\]
naturally in $X,Y$.\qed
\end{lemma}

%

\subsection{Bar complex}
\label{ss:bar complex}

\begin{definition}\label{def:bar cx}
Let $\CS$ be a dg category.  The \emph{2-sided bar complex of $\CS$} is the bimodule $\Bar(\CS)\in \Bim_{\CS,\CS}$ defined by
\[
\Bar(\CS):=\tw_\d\left(\bigoplus_{r\geq 0} q^{-r\iota} \bigoplus_{X_0,X_1,\ldots,X_r\in \Obj(\CS)} \ket{X_0}\otimes \braket{X_0}{X_1}\otimes\cdots\otimes \braket{X_{r-1}}{X_r}\otimes \bra{X_r}\right)
\]
where the twist $\d$ is the usual bar differential, which will be described in a moment.  Simple tensors in the bar complex will be denoted by $[f_0,\ldots,f_{r+1}]:=f_0\otimes\cdots \otimes f_{r+1}$, and the bimodule structure is given by
\[
g\cdot[f_0,f_1,\ldots,f_r,f_{r+1}]\cdot g' = (-1)^{r\ip{\iota,|g|}}[g\circ f_0,f_1,\ldots,f_r,f_{r+1}\circ g'],
\]
and the differential $d([f_0,\ldots,f_{r+1}])$ is a sum of terms of two types:
\begin{enumerate}
\item internal differential applied to $i$-th term: $(-1)^{r+k_0+\cdots+k_{i-1}} [f_0,\ldots,d(f_i),\ldots,f_{r+1}]$ ($1\leq i\leq r+1$),
\item bar differential $\d$ applied to $(i,i+1)$ terms: $(-1)^i [f_0,\ldots,f_i\circ f_{i+1},\ldots,f_{r+1}]$ ($0\leq i\leq r$),
\end{enumerate}
where $k_i=\ip{\iota,|f_i|}$ for $0\leq i\leq r+1$.  The \emph{counit} of $\Bar(\CS)$ is the bimodule map $\e\colon \Bar(\CS)\rightarrow \CS$ defined by
\[
\braCket{Y}{\Bar(\CS)}{Y'}\rightarrow \braCket{Y}{\CS}{Y'}  \colon  \begin{cases} [f_0,f_1]\mapsto f_0\circ f_1  & \\ [f_0,\ldots,f_{r+1}]\mapsto 0 &\text{ if }r>0\end{cases}
\]
\end{definition}

We will also use an alternate sign convention for the bar complex of $\CS$.

\begin{definition}\label{def:bar complex 2}
Given an element $[f_0,\ldots,f_{r+1}]\in \Bar(\CS)(Y,Y')$ with $k_u:=\ip{\iota,f_u}$ for $0\leq u\leq r+1$, let
\begin{equation}\label{eq:bar complex 2}
[f_0,\ldots,f_{r+1}]' := (-1)^{\sigma(k_0,\ldots,k_r)} [f_0,\ldots,f_{r+1}]
\end{equation}
where $\sigma(k_0,\ldots,k_{r+1})=\sum_{u=0}^r (r+1-i)k_u$. 
\end{definition}

\begin{remark}
One can think of the sign convention \eqref{eq:bar complex 2} in the following way.  First, in the standard bar complex one should imagine that there are $r$ copies of the shift $q^{\iota}$ sitting to the left of $[f_0,\ldots,f_{r+1}]$, as in $[f_0,\ldots,f_{r+1}]=\textcolor{black}{q^{r\iota}}f_0\otimes \cdots \otimes f_{r+1}]$.  In the alternate, ``primed'', notation we think of the shifts $q^\iota$ as distributed as in $[f_0,\ldots,f_{r+1}]'=\textcolor{black}{q^{-\iota}} f_0\otimes \textcolor{black}{q^\iota} f_1\otimes\cdots \otimes \textcolor{black}{q^\iota} f_{r+1}$.
\end{remark}
The following is straightforward:

\begin{lemma}\label{lemma:bar complex 2}
In terms of \eqref{eq:bar complex 2} the bimodule structure is given by
\[
g\cdot [f_0,f_1,\ldots,f_r,f_{r+1}]'\cdot g' = (-1)^{\ip{\iota,|g|}} [g\circ f_0,f_1,\ldots,f_r,f_{r+1}\circ g']'
\]
and the differential $d([f_0,\ldots,f_{r+1}]')$ is a sum of terms of two types:
\begin{enumerate}
\item $(-1)^{u+k_0+\cdots+k_{u-1}+1} [f_0,\ldots,d(f_u),\ldots,f_{r+1}]'$ for $1\leq u\leq r+1$,
\item $(-1)^{u+k_0+\cdots+k_u} [f_0,\ldots,f_u\circ f_{u+1},\ldots,f_{r+1}]'$ for $0\leq u\leq r$.
\end{enumerate}
\end{lemma}
\begin{proof}
This is a consequence of the following:
\[
\sigma(\ldots,k_u+1,\ldots) = (r+1-u)+\sigma(\ldots,k_u,\ldots)
\]
\begin{align*}
\sigma(\ldots,k_u+k_{u+1},\ldots)
&=\sum_{v=0}^u (r-v)k_v + \sum_{v=u+1}^r (r+1-v)k_v\\
&=-(k_0+\cdots+k_u)+\sigma(\ldots,k_u,k_{u+1},\ldots)
\end{align*}
\end{proof}

\subsection{Relative bar complex}
\label{ss:relative bar}
The direct sum occuring in the definition of $\Bar(\DS)$ is quite large, and can be tamed a bit by working with the bar complex of $\DS$ relative to a chosen subset of objects of $\DS$, defined next.

\begin{definition}\label{def:relative bar cx}
Let $\DS$ be a dg category, and let $\XS\subset \Obj(\DS)$.  The \emph{2-sided bar complex of $\DS$ relative to $\XS$} is the bimodule $\Bar(\DS,\XS)\in \Bim_{\DS,\DS}$ defined by
\[
\Bar(\DS,\XS):=\tw_\d\left(\bigoplus_{r\geq 0} q^{-r\iota} \bigoplus_{X_0,X_1,\ldots,X_r\in \Obj(\DS)} \ket{X_0}\otimes \braket{X_0}{X_1}\otimes\cdots\otimes \braket{X_{r-1}}{X_r}\otimes \bra{X_r}\right)
\]
where the twist $\d$ and the bimodule structure are as in Definition \ref{def:bar cx}.
\end{definition}
The following is an easy observation.

\begin{lemma}
Let $\CS\subset \DS$ be a full dg subcategory, and let $\XS=\Obj(\CS)$.  Then there is an isomorphism of bimodules
\[
\DS\otimes_\CS \Bar(\CS)\otimes_\CS\DS \to \Bar(\DS,\XS)
\]
sending $g\otimes [f_0,\ldots,f_{r+1}]\otimes g' \cdot [f_0,\ldots,f_{r+1}]$.\qed
\end{lemma}

We will set up some notation for relative bar complexes that will be used throughout.  Given a dg category $\DS$ and a full subcategory $\CS\subset \DS$, we will use bold font to denote objects and morphisms in $\DS$ and standard font to denote objects and morphisms in $\CS$.  A typical simple tensor in $\Bar(\DS)$ will denoted by $[\mathbf{g},\mathbf{f}_1,\ldots,\mathbf{f}_r,\mathbf{g}']$ corresponding to a sequence of composable morphisms
\begin{equation}\label{eq:sequenceBold}
\begin{tikzpicture}
\node (a) at (0,0) {$\mathbf{Y}$};
\node (b) at (2,0) {$\mathbf{X}_0$};
\node (c) at (4,0) {$\cdots$};
\node (d) at (6,0) {$\mathbf{X}_r$};
\node (e) at (8,0) {$\mathbf{Y}'$};
\path[-stealth]
(b) edge node[above] {$\mathbf{g}$} (a)
(c) edge node[above] {$\mathbf{f}_1$} (b)
(d) edge node[above] {$\mathbf{f}_{r}$} (c)
(e) edge node[above] {$\mathbf{g}'$} (d);
\end{tikzpicture}
\end{equation}
When we wish to treat the extremal and internal terms on equal footing, we will also write
\begin{equation}\label{eq:extremal terms}
\XB_{-1}:=\YB\, , \qquad \XB_{r+1}:=\YB' \, , \qquad \mathbf{f}_0:=\mathbf{g}\, , \qquad  \mathbf{f}_{r+1}:=\mathbf{g}',
\end{equation}
A typical simple tensor in $\DS\otimes_\CS\Bar(\CS)\otimes_\CS \DS$ will be denoted by $\mathbf{g}\cdot [f_0,\ldots, f_{r+1}]\cdot \mathbf{g}'$, corresponding to a sequence of composable morphisms
\begin{equation}\label{eq:sequence in FCF}
\begin{tikzpicture}[anchorbase]
\node (a) at (0,0) {$\mathbf{Y}$};
\node (b) at (2,0) {${X}_{-1}$};
\node (c) at (4,0) {${X}_0$};
\node (d) at (6,0) {$\cdots$};
\node (e) at (8,0) {${X}_r$};
\node (f) at (10,0) {${X}_{r+1}$};
\node (g) at (12,0) {$\mathbf{Y}'$};
\path[-stealth]
(b) edge node[above] {$\mathbf{g}$} (a)
(c) edge node[above] {${f}_0$} (b)
(d) edge node[above] {${f}_1$} (c)
(e) edge node[above] {${f}_{r}$} (d)
(f) edge node[above] {${f}_{r+1}$} (e)
(g) edge node[above] {$\mathbf{g}'$} (f);
\end{tikzpicture}
\end{equation}
where the $f_i$ are morphisms in $\CS$.  Note that tensoring over $\CS$ means that the generators $\mathbf{g}\cdot [f_0,\ldots, f_{r+1}]\cdot \mathbf{g}'$ are subject to the relations
\begin{equation}\label{eq:gfffg relation}
\mathbf{g}\cdot [f_0,\ldots, f_{r+1}]\cdot \mathbf{g}' = (-1)^{r\ip{\iota,|f_0|}} (\mathbf{g}\circ f_0)\cdot [\id_{X_0},f_1,\ldots,f_r,\id_{X_r}]\cdot (f_{r+1}\circ \mathbf{g}')
\end{equation}

The sign comes from the twisting of the left $\CS$-action which arises from the shift $q^{-r\iota}$, which is applied to $[f_0,\ldots,f_{r+1}]$ when forming the bar complex of $\CS$.

\begin{definition}\label{def:generation}
Let $\XS\subset \Obj(\DS)$ be subset of objects.  We say that $\XS$ \emph{generates} $\DS$ if every object of $\DS$ is homotopy equivalent to a finite one-sided twisted complex constructed from objects in $\XS$.
\end{definition}

\begin{theorem}\label{thm:bar and generation}
If $\XS\subset \Obj(\DS)$ is a set of objects which generates $\DS$, then $\Bar(\DS)\simeq \Bar(\DS,\XS)$.  In particular, if $\CS$ is a dg category and $\EB$ is one of the envelope operations $\EB\in \{\SB,\AB,\pretr\}$ then we have
\begin{equation}\label{eq:bar of envelop equiv}
\Bar(\EB(\CS)) \simeq \Bar(\EB(\CS), \Obj(\CS)).
\end{equation}
\end{theorem}

The proof will utilize the following unique characterization of $\Bar(\DS,\XS)$.  To state it, first recall the counit map $\e\colon \Bar(\DS)\to \DS$.   Now $\Bar(\DS,\XS)$ is a sub-bimodule of $\Bar(\DS)$, and inherits a counit map $\e'\colon \Bar(\DS,\Obj(\CS))\to \DS$ by restriction.

\begin{definition}\label{def:bar axioms}
Let $\DS$ be a dg category with $\XS\subset \Obj(\DS)$ a set of objects.  Let $(B,\e_B)$ be a pair in which $B$ is a $(\DS,\DS)$-bimodule and $\e_B\colon B\to \one_\DS$ is a bimodule map.  We say that the pair $(B,\e_B)$ satisfies property $I_{\XS}$ (resp.~$K_{\XS}$) if:
\begin{itemize}
\item ($I_{\XS}$) $B$ is a one-sided twisted complex constructed from bimodules $\ket{X_0}\otimes M\otimes \bra{X_r}$ where $X_0,X_r\in \XS$ and $M\in \dgMod{\Gamma}{\k}$.  
\item($K_{\XS}$)  $\bra{Y}\otimes_\DS \Cone(\e)\simeq 0 \simeq \Cone(\e)\otimes_\DS \ket{Y'}$ for all objects $Y,Y'\in \XS$.
\end{itemize}
\end{definition}

\begin{lemma}\label{lemma:characterization of bar}
The pair $(\Bar(\DS,\XS),\e')$ is uniquely characterized up to homotopy equivalence by the fact that it satisfies $I_{\XS}$ and $K_{\XS}$.
\end{lemma}
\begin{proof}
Exercise.
\end{proof}

\begin{proof}[Proof of Theorem \ref{thm:bar and generation}]
Assume that $\XS\subset \Obj(\DS)$ is a set of objects which generates $\DS$.  Since $\XS$ generates $\DS$, property $I_{\XS}$ implies property $I_{\Obj(\DS)}$.  We claim that property $K_{\XS}$ implies $K_{\Obj(\DS)}$.  Indeed, since $\XS$ generates $\DS$, each Yoneda module $\bra{\mathbf{Y}}$ is homotopy equivalent to a finite one-sided twisted complex
\[
\bra{\mathbf{Y}}\simeq \tw_{\b}\left(\bigoplus_{b\in B} q^{j_b}  \bra{Y_b}\right)
\]
with $Y_b\in \XS$.  Tensoring with $\Cone(\e')$ yields
\[
\bra{\mathbf{Y}} \otimes_\DS \Cone(\e')\simeq  \tw_{\b\otimes \id}\left(\bigoplus_{b\in B} q^{j_b}  \bra{Y_b}\otimes_\DS \Cone(\e')\right)
\]
which is contractible since each $\bra{Y_b} \Cone(\e')$ is contractible by (2').  A similar argument deals with placing $\ket{\mathbf{Y}'}$ on the right. This completes the proof.
\end{proof}

We wish to have explicit formulas for the homotopy equivalence from $\Bar(\DS)$ to $\Bar(\DS,\CS)$ in the special case where $\DS$ is one of the envelope operations $\{\SB,\AB,\pretr\}$ applied to $\CS$.  We do this next.

\subsection{The bar complex of $\SB(\CS)$}
\label{ss:susp and bar}
We consider the homotopy equivalence \eqref{eq:bar of envelop equiv}) in the case of $\EB=\SB$.  Define $\Xi_\SB\colon \Bar(\SB(\CS))\to \SB(\CS)\otimes_\CS\Bar(\CS)\otimes_\CS \SB(\CS)$ to be the map which sends a sequence of composable morphisms
\begin{equation}\label{eq:S(C) sequenceBold}
\begin{tikzpicture}[baseline=-.1cm]
\node (a) at (0,0) {$q^{i_{-1}}X_{-1}$};
\node (b) at (2.5,0) {$q^{i_0}X_0$};
\node (c) at (5,0) {$\cdots$};
\node (d) at (7.5,0) {$q^{i_r}X_r$};
\node (e) at (10,0) {$q^{i_{r+1}}X_{r+1}$};
\path[-stealth]
(b) edge node[above] {$(f_0)^{i_{-1}}_{i_0}$} (a)
(c) edge node[above] {$(f_1)^{i_0}_{i_1}$} (b)
(d) edge node[above] {$(f_r)^{i_{r-1}}_{i_r}$} (c)
(e) edge node[above] {$(f_{r+1})^{i_r}_{i_{r+1}}$} (d);
\end{tikzpicture}
\end{equation}
to $(-1)^{\ip{\iota, i_{-1}}}$ times the sequence
\begin{equation}\label{eq:S(C) sequence in FCF}
\begin{tikzpicture}[baseline=-.1cm]
\node (a) at (0,0) {$q^{i_{-1}}X_{-1}$};
\node (b) at (2,0) {${X}_{-1}$};
\node (c) at (3.5,0) {${X}_0$};
\node (d) at (5,0) {$\cdots$};
\node (e) at (6.5,0) {${X}_r$};
\node (f) at (8,0) {${X}_{r+1}$};
\node (g) at (10,0) {$q^{i_{r+1}}X_{r+1}$};
\path[-stealth]
(b) edge node[above] {$\id^{i_{-1}}_0$} (a)
(c) edge node[above] {${f}_0$} (b)
(d) edge node[above] {${f}_1$} (c)
(e) edge node[above] {${f}_{r}$} (d)
(f) edge node[above] {${f}_{r+1}$} (e)
(g) edge node[above] {$\id^0_{i_{r+1}}$} (f);
\end{tikzpicture}
\end{equation}
In other words, letting $\mathbf{f}_u=(f_u)^{i_{u-1}}_{i_u}$, $q^jY=q^{i_{-1}}X_{-1}$, and $q^{j'}Y'=q^{i_{r+1}}X_{r+1}$, we have
\begin{equation}\label{eq:PhiS fff}
\Xi_{\SB}([\mathbf{f}_0,\mathbf{f}_1,\ldots,\mathbf{f}_r,\mathbf{f}_{r+1}]) 
= (-1)^{r\ip{\iota,j}} \phi_{Y,j} \cdot[f_0,f_1,\ldots,f_{r+1}]\cdot (\phi_{Y',j'})\inv
\end{equation}

%
%
%
%
%
%

\begin{lemma}\label{lemma:PhiS}
Equation \eqref{eq:PhiS fff} defines a homotopy equivalence $\Xi_\SB\colon \Bar(\SB(\CS))\to \SB(\CS)\otimes_\CS\Bar(\CS)\otimes_\CS \SB(\CS)$.
\end{lemma}
\begin{proof}
Let us introduce some shorthand. For $0\leq u\leq r$ we let $k_u=\ip{\iota,|f_u|}$ and $\mathbf{k}_u=\ip{\iota,|\mathbf{f}_u|}$. For $-1\leq u\leq r+1$ we let $m_u:=\ip{\iota,i_u}$.

To see that $\Xi_{\SB}$ is a bimodule map, recall that the shift $q^{-r\iota}M$ twists the left action of any bimodule $M\in \Bim_{\CS,\CS}$ by a sign, and compute:
\begin{align*}
\Xi_{\SB}([\mathbf{f}_0,\mathbf{f}_1,\ldots,\mathbf{f}_r,\mathbf{f}_{r+1}])
&=(-1)^{rm_{-1}} \phi_{Y,j} \cdot[f_0,f_1,\ldots,f_{r+1}]\cdot (\phi_{Y',j'})\inv\\
&=(-1)^{rm_{-1}+rk_0} (\phi_{Y,j}\circ f_0)\cdot[\id_{X_0},f_1,\ldots,f_r,\id_{X_r}]\cdot (f_{r+1}\circ (\phi_{Y',j'})\inv)\\
&= (-1)^{r\mathbf{k}_0+rm_0} (\mathbf{f}_0\circ \phi_{X_0,i_0})\cdot [\id_{X_0},f_1,\ldots,f_{r+1},\id_{X_r}]\cdot ((\phi_{X_r,i_r})\inv\circ \mathbf{f}_{r+1})\\
&= \mathbf{f}_0\cdot \Xi_{\SB}([\id_{\XB_{0}},\mathbf{f}_1,\ldots,\mathbf{f}_r,\id_{\XB_{r}}]) \cdot \mathbf{f}_{r+1}.
\end{align*}
In the first line we used the definition of $\Xi_\SB$.  In the second line we used the relation \eqref{eq:gfffg relation}. In the third we used the definition of $\mathbf{f}_u=(f_u)^{i_{u-1}}_{i_u}$, and $\mathbf{k}_0=m_{-1}+k_0-m_0$ to rewrite the sign. In the last line we used the definition of $\Xi_\SB$ together with sign-twisted left action.


To see that $\Xi_\SB$ is closed, note that $(-1)^{rm_{-1}} \Xi_{\SB}\circ d([\mathbf{f}_0,\mathbf{f}_1,\ldots,\mathbf{f}_{r+1}])$ is sum of terms
\begin{align*}
(-1)^{r+\mathbf{k}_0+\cdots+\mathbf{k}_{u-1}+m_{u-1}} &\  \phi_{Y,j}\cdot[f_0,\ldots,d(f_u),\ldots,f_{r+1}]\cdot (\phi_{Y',j'})\inv &&\text{ for } 0\leq u\leq r+1\\
(-1)^{u} &\  \phi_{Y,j}\cdot[f_0,\ldots,f_u\circ f_{u+1},\ldots,f_{r+1}]\cdot (\phi_{Y',j'})\inv &&\text{ for } 0\leq u\leq r
\end{align*}
and $(-1)^{rm_{-1}} d\circ \Xi_{\SB}([\mathbf{g},\mathbf{f}_1,\ldots,\mathbf{f}_r,\mathbf{g}']$ is sum of terms
\begin{align*}
(-1)^{r+m_{-1}+k_0+\cdots+k_{u-1}} &\ \phi_{Y,j}\cdot [f_0,\cdots ,d(f_u),\ldots,f_{r+1}]\cdot (\phi_{Y',j'})\inv && \text{ for }0\leq u\leq r+1\\
(-1)^{u}&\  \phi_{Y,j}\cdot[f_0,\ldots,f_u\circ f_{u+1},\ldots,f_{r+1}]\cdot (\phi_{Y',j'})\inv && \text{ for }0\leq u\leq r
\end{align*}
The signs agree since
\[
\mathbf{k}_0+\cdots+\mathbf{k}_{u-1}+m_{u-1} = m_{-1}+k_0+\cdots+k_{u-1},
\]
which holds since $\mathbf{k}_v=m_{v-1}+k_v-m_v$ for all $v$.  This completes the proof that $\Xi_\SB$ is a chain map.

Finally, to see that $\Xi_\SB$ is a homotopy equivalence, it suffices by Proposition \ref{prop:bar props}  to show that $\Xi_\SB$ is compatible with the counits, which is clear.
\end{proof}


\subsection{The bar complex of the additive envelope}
\label{ss:add and bar}
We consider the homotopy equivalence \eqref{eq:bar of envelop equiv} in the case of $\EB=\AB$. 
Suppose we have an element $[\mathbf{f}_0,\ldots,\mathbf{f}_{r+1}]$ in the bar complex $\Bar(\AB(\CS))(\YB,\YB')$, regarded as a sequence of composable morphisms
\[
\begin{tikzpicture}[baseline=-.1cm]
\node (a) at (0,0) {$\XB_{-1}$};
\node (b) at (2.5,0) {$\XB_0$};
\node (c) at (5,0) {$\cdots$};
\node (d) at (7.5,0) {$\XB_r$};
\node (e) at (10,0) {$\XB_{r+1}$};
\path[-stealth]
(b) edge node[above] {$\mathbf{f}_0$} (a)
(c) edge node[above] {$\mathbf{f}_1$} (b)
(d) edge node[above] {$\mathbf{f}_r$} (c)
(e) edge node[above] {$\mathbf{f}_{r+1}$} (d);
\end{tikzpicture}
\]
where $\YB=\XB_{-1}$ and $\YB'=\XB_{r+1}$.  Write each object $\XB_u$ above as $\XB_u = \bigoplus_{a\in A_u} X_a$.  Further, for each $a\in A_u$, let $\sigma_a\colon X_a\hookrightarrow \XB_u$ and $\pi_e\colon \XB_u\twoheadrightarrow X_a$ be the inclusion and projection of the direct summand, and write the components of $\mathbf{f}_u$ as $f_{a,a'} := \pi_a\circ \mathbf{f}_u\circ \sigma_{a'}$ for each $(a,a')\in A_{u-1}\times A_u$.

Now we may define $\Xi_{\AB(\CS)}\colon \Bar(\AB(\CS))\to \AB(\CS)\otimes_\CS \Bar(\CS)\otimes_\CS \AB(\CS)$ to be the map sending the sequence of composable morphisms above to the sum over all indices $(a_{-1},\ldots,a_{r+1})\in A_{-1}\times \cdots\times A_{r+1}$ of
\[
\begin{tikzpicture}[baseline=-.1cm]
\node (a) at (0,0) {$\XB_{-1}$};
\node (b) at (2.5,0) {${X}_{a_{-1}}$};
\node (c) at (4.5,0) {${X}_{a_0}$};
\node (d) at (6.5,0) {$\cdots$};
\node (e) at (8.5,0) {${X}_{a_r}$};
\node (f) at (10,0) {${X}_{a_{r+1}}$};
\node (g) at (12,0) {$\XB_{r+1}$};
\path[-stealth]
(b) edge node[above] {$\sigma_{a_{-1}}$} (a)
(c) edge node[above] {${f}_{a_{-1},a_0}$} (b)
(d) edge node[above] {${f}_{a_{0},a_1}$} (c)
(e) edge node[above] {${f}_{a_{r-1},a_r}$} (d)
(f) edge node[above] {${f}_{r+1}$} (e)
(g) edge node[above] {$\pi_{a_{r+1}}$} (f);
\end{tikzpicture}.
\]
That is to say, we defne $\Xi_\AB$ by the formula
\begin{equation}\label{eq:PhiA fff}
\Xi_\AB\Big([\mathbf{f}_0,\mathbf{f}_1,\ldots,\mathbf{f}_{r+1}]\Big) := \sum_{a_{-1},\ldots,a_{r+1}} \sigma_{a_{-1}}\cdot \Big[f_{a_{-1},a_0}, f_{a_0,a_1}, \ldots, f_{a_r,a_{r+1}} \Big]\cdot \pi_{a_{r+1}}.
\end{equation}

\begin{lemma}\label{lemma:PhiA}
Equation \eqref{eq:PhiA fff} defines a homotopy equivalence $\Xi_\AB\colon \Bar(\AB(\CS))\to \AB(\CS)\otimes_\CS\Bar(\CS)\otimes_\CS \AB(\CS)$.
\end{lemma}
\begin{proof}
It is easy to see that $\Xi_\AB$ is a chain map and commutes with the bimodule structures.  Moreoover, $\Xi_\AB$ is compatible with the counits, hence is a homotopy equivalence by Proposition \ref{prop:bar props}.
\end{proof}


\subsection{The bar complex of the twisted envelope}
\label{ss:twist and bar}

In general the homotopy equivalence \eqref{eq:bar of envelop equiv} is false for $\EB=\Tw$, the twisted envelope.   The main issue is that the formula that one wishes to use to define $\Xi_{\Tw}$ is actually an ill-defined infinite sum.  One strategy for overcoming this problem is to work some completed version of the bar complex of $\CS$.  In this section we show how to construct a chain map $\Xi_{\Tw}$ using this approach, but we warn that $\Xi_{\Tw}$ will generally not be a homotopy equivalence.

\begin{definition}
Given a dg category $\CS$, let $\overline{\Bar}(\CS)\in \Bim_{\CS,\CS}$, the \emph{completed bar complex of $\CS$}, be defined by
\[
\overline{\Bar}(\CS)(Y,Y') = \tw_\d\left(\prod_{\substack{r\geq 0 \\ X_0,\ldots,X_r}} q^{-r\iota}  \braket{Y}{X_0}\otimes \braket{X_0}{X_1}\otimes\cdots\otimes \braket{X_{r-1}}{X_r}\otimes \braket{X_r}{Y'}\right)
\]
where the productis over all finite sequences of objects $X_0,\ldots,X_r\in \Obj(\CS)$, and  $\d$ is the usual bar differential.
\end{definition}

\begin{remark}
The completed bar complex has has no counit or coalgebra structure, and is generally not well-behaved since it is not the colimit of its $r\leq R$ truncations.
\end{remark}


Suppose we have a sequence \eqref{eq:sequenceBold} of composable morphisms in $\Tw(\CS)$.  Also recall the alternate notation for the extremal terms \eqref{eq:extremal terms}.  Let us write
\[
\mathbf{X}_u = \tw_{\a_u}(X_u)
\]
for $0\leq u\leq r+1$, where $X_u\in \Obj(\CS)$. We will also write $\XB_{-1}=\YB=\tw_\b(Y)$ and $\XB_{r+1}=\YB'=\tw_{\b'}(Y')$.  For each $0\leq u\leq r+1$, let $f_u\in \Hom_{\CS}(X_u,X_{u-1})$ denote the map underlying $\mathbf{f}_u\in \Hom_{\Tw(\CS)}(\XB_u,\XB_{u-1})$ (in other words, $\mathbf{f}_u = (f_u)^{\a_{u-1}}_{\a_u}$ using notation from Definition \ref{def:twisted envelope}).

For each $r\in \Z_{\geq 0}$ and each sequence of non-negative integers $n_0,\ldots,n_r\in \Z_{\geq 0}$, let us define a map
\begin{equation}
\begin{tikzpicture}[baseline=-1cm]
\node (a) at (0,0)
{$q^{-(r+n_0+\cdots+n_r)\iota} \braket{{X}_{-1}}{{X}_0} \otimes \braket{{X}_{0}}{{X}_0}^{\otimes n_0}\otimes \braket{{X}_{0}}{{X}_1}\otimes \cdots  \otimes \braket{{X}_{r}}{{X}_r}^{\otimes n_r} \otimes \braket{{X}_{r}}{{X}_{r+1}}$};
\node (b) at (0,-1.5) {$q^{-r\iota}\braket{\mathbf{X}_{-1}}{\mathbf{X}_0} \otimes \braket{\mathbf{X}_0}{\mathbf{X}_1}\otimes\cdots \otimes \braket{\mathbf{X}_r}{\mathbf{X}_{r+1}}$};
\path[-stealth]
(b) edge node[left] {$\Xi_{n_0,\ldots,n_r}$}(a);
\end{tikzpicture}
\end{equation}
by the formula
\begin{equation}\label{eq:expanded bar}
\Xi_{n_0,\ldots,n_r}([\mathbf{f}_0,\ldots,\mathbf{f}_{r+1}]'):=  
\left[f_0, \underbrace{\a_0,\ldots,a_0}_{n_0},f_1,\underbrace{\a_1,\ldots,a_1}_{n_1},\ldots,f_r,\underbrace{\a_r,\ldots,\a_r}_{n_r}, f_{r+1}\right]'
\end{equation}
where we are using the sign convention from Definition \ref{def:bar complex 2}.

\begin{lemma}\label{lemma:PhiTw}
We have a well-defined chain map $\Xi_{\Tw}\colon \Bar(\Tw(\CS))\to \Tw(\CS)\otimes_\CS\overline{\Bar}(\CS)\otimes_\CS \Tw(\CS)$ defined by
\begin{equation}\label{eq:PhiTw}
\Xi_{\Tw}([\mathbf{f}_0,\ldots,\mathbf{f}_{r+1}]'):=\psi_{\YB}\cdot\left(\sum_{r\geq 0}\sum_{n_0,\ldots,n_r\geq 0}  \Xi_{n_0,\ldots,n_r}([\mathbf{f}_0,\ldots,\mathbf{f}_{r+1}]')\right)\cdot \psi_{\YB'}\inv
\end{equation}
where $\psi_{\YB}=\psi_{Y,\b}\colon Y\to  \tw_\b(Y)$ denotes the canonical map, given by $\id_Y$.
\end{lemma}

Let us introduce notation for bookkeeping purposes.  Fix $n_0,\ldots,n_r$ and fix $[\mathbf{f}_0,\ldots,\mathbf{f}_{r+1}]'$.  Let $s:=r +\sum_{u=0}^r n_u$, and denote denote the $v$-th term of \eqref{eq:expanded bar} by $h_v$ (for $0\leq v\leq s+1$), so that
\[
\Xi_{n_0,\ldots,n_r}([\mathbf{f}_0,\ldots,\mathbf{f}_{r+1}]')=[h_0,\ldots,h_{s+1}]'.
\]
To be precise, for each index $0\leq v\leq s+1$ let $u$ be the largest index for which $v-(u+n_0+\cdots+n_{u-1})\geq 0$.  Then $v-(u+1+n_0+\cdots+n_u)<0$ implies $0\leq v-(u+n_0+\cdots+n_{u-1})<n_u+1$.  Thus for each index $v$ we have the following dichotomy: either $v=u+n_0+\cdots+n_{u-1}$ for some index $u$, or $v-(u+n_0+\cdots+n_{u-1})\in \{1,\ldots,n_u\}$ for some (unique) $0\leq u\leq r+1$, which we record with a binary sequence $\nu=(\nu_0,\ldots,\nu_{s+1})\in \{0,1\}^{s+2}$, defined by
\[
\nu_v := \begin{cases}
1 & \text{ if } v=u+n_0+\cdots+n_{u-1}\\
0 & \text{ if } v-(u+n_0+\cdots+n_{u-1})\in \{1,\ldots,n_u\}
\end{cases}.
\]
We now define
\[
h_v := \begin{cases}
f_u & \text{ if } v=u+n_0+\cdots+n_{u-1}\\
\a_u & \text{ if } v-(u+n_0+\cdots+n_{u-1})\in \{1,\ldots,n_u\}
\end{cases}
\]
for $0\leq v\leq s+1$.   Let $k_u:=\ip{\iota,|f_u|}$ for $0\leq u\leq r+1$, and let $l_v:=\ip{\iota,h_v}$ for $0\leq v\leq s+1$.

We make some easy observations about the binary sequence $\nu$.  First, $\nu_0=1=\nu_{s+1}$ always, corresponding to the fact that $h_0=f_0$ and $h_{s+1}=f_{r+1}$.  Second the total number of occurences of `1' in $\nu$ is $r+2$, u.e.~$\sum_{v=0}^{s+1} \nu_v=r+2$.  Finally, if $\nu_v=1$ then $h_v=f_u$, where $u$ is the number of indices $0\leq v'<v$ such that $\nu_{v'}=1$.

%

The following lemma will be useful.

\begin{lemma}\label{lemma:signs oh my}
Given an index $0\leq v\leq s+1$, let $0\leq u\leq r+1$ be the largest index such that $u+n_0+\cdots+n_{u-1}\leq v$, so that $h_v=f_u$ or $h_v=\a_u$. In either case
\[
l_0+\cdots+l_{v} =v-u+k_0+\cdots+k_{u}.
\]
\end{lemma}
\begin{proof}
The left-hand side has a term of the form $k_{u'}$  for each index $u'\leq u$. The remaining terms are all equal to $+1$, with the number of these being equal to $(v+1)-(u+1)$.
\end{proof}

\begin{proof}[Proof of Lemma \ref{lemma:PhiTw}]
Observe that $\Xi_{\Tw}\circ d([\mathbf{f}_0,\ldots,\mathbf{f}_{r+1}]')$ is a sum of terms of various types.  First, we have the terms which come from applying the internal differential $d_{\Tw(\CS)}$ to one of the $\mathbf{f}_u$ (for $0\leq u\leq r+1$).  Each one of these is itself a sum of three terms:
\begin{subequations}
\begin{align}
\label{eq:diff1:df}(-1)^{u+k_0+\cdots+k_{u-1}+1}  &\ \psi_{\YB}\cdot[\ldots, d(f_u),\ldots]'\cdot \psi_{\YB'}\inv\\
\label{eq:diff1:af}(-1)^{u+k_0+\cdots+k_{u-1}+1}  &\ \psi_{\YB}\cdot[\ldots, \a_{u-1}\circ f_u,\ldots]'\cdot \psi_{\YB'}\inv\\
\label{eq:diff1:fa}(-1)^{u+k_0+\cdots+k_u}  &\ \psi_{\YB}\cdot[\ldots, f_u\circ \a_u,\ldots]'\cdot \psi_{\YB'}\inv
\end{align}
Next we have the terms coming from the bar differential, composing $\mathbf{f}_u$, $\mathbf{f}_{u+1}$:
\begin{align}
\label{eq:diff1:ff}(-1)^{u+k_0+\cdots+k_u}  &\  [\ldots,f_u\circ f_{u+1},\ldots]'\cdot \psi_{\YB'}\inv
\end{align}
\end{subequations}

Now we consider the terms which appear in $d\circ \Xi_{\Tw}([\mathbf{f}_0,\ldots,\mathbf{f}_{r+1}])$.  First, we have $d_{\Tw(\CS)}$ applied to one of the extremal morphisms $\psi_{\YB}$ or $\Psi_{\YB'}\inv$:
\begin{subequations}
\begin{align}
\label{eq:diff2:bf} +1   &\  (\psi_{\YB}\circ \b)\cdot [\ldots\ldots]'\cdot \psi_{\YB'}\inv\\
\label{eq:diff2:fb}(-1)^{s+1+l_0+\cdots+l_{s+1}}   &\  \psi_{\YB}\cdot[\ldots\ldots]'\cdot (\b'\circ \psi_{\YB'}\inv)
\end{align}
Next, we have terms coming from applying $d_\CS$ to one of the internal morphisms $h_v$.  There are two cases, dending on whether $\nu_v=1$ or $\nu_v=0$:
\begin{align}
\label{eq:diff2:df}(-1)^{v+l_0+\cdots+l_{v-1}+1}   &\  \psi_{\YB}\cdot[\ldots, d(f_u),\ldots]'\cdot\psi_{\YB'}\inv & \text{if $h_v=f_u$}\\
\label{eq:diff2:da}(-1)^{v+l_0+\cdots+l_{v-1}+1}   &\  \psi_{\YB}\cdot[\ldots, d(\a_u),\ldots]'\cdot \psi_{\YB'}\inv& \text{if $h_v=\a_u$}
\end{align}
Finally we have terms from applying the bar differential, which composes $h_v$, $h_{v+1}$.  There are four cases, depending on $(\nu_v,\nu_{v+1})\in \{0,1\}^2$.
\begin{align}
\label{eq:diff2:aa}(-1)^{v+l_0+\cdots+l_v}   &\ \psi_{\YB}\cdot[\ldots,\a_u\circ \a_u,\ldots]'\cdot \psi_{\YB'}\inv& \text{ if }(\nu_v,\nu_{v+1})=(0,0)\\
\label{eq:diff2:af}(-1)^{v+l_0+\cdots+l_v}   &\  \psi_{\YB}\cdot[\ldots,\a_{u-1}\circ f_u,\ldots]'\cdot \psi_{\YB'}\inv& \text{ if }(\nu_v,\nu_{v+1})=(0,1)\\
\label{eq:diff2:fa}(-1)^{v+l_0+\cdots+l_v}   &\  \psi_{\YB}\cdot[\ldots,f_u\circ \a_u,\ldots]'\cdot \psi_{\YB'}\inv& \text{ if }(\nu_v,\nu_{v+1})=(1,0)\\
\label{eq:diff2:ff}(-1)^{v+l_0+\cdots+l_v}   &\  \psi_{\YB}\cdot[\ldots,f_u\circ f_{u+1},\ldots]'\cdot \psi_{\YB'}\inv& \text{ if }(\nu_v,\nu_{v+1})=(1,1)
\end{align}
\end{subequations}
Terms of type \eqref{eq:diff2:aa} and \eqref{eq:diff2:da} cancel because we have the Maurer-Cartan equation $d(\a_u)+\a_u^2=0$.  

Terms of the form $[\ldots,d(f_u),\ldots]'$ occur in \eqref{eq:diff1:df} and \eqref{eq:diff2:df}, with signs
\[
(-1)^{u+k_0+\cdots+k_{u-1}+1} = (-1)^{v+l_0+\cdots+l_{v-1}+1}.
\]

Terms of the form $[\ldots,\a_{u-1}\circ f_u,\ldots]'$ with $u\neq 0$ occur in \eqref{eq:diff1:af} and \eqref{eq:diff2:af}), with signs
\[
(-1)^{u+k_0+\cdots+k_{u-1}+1} \text{ respectively } (-1)^{v+l_0+\cdots+l_v}.
\]
By Lemma \ref{lemma:signs oh my}, $h_{v+1}=f_u$ implies $l_0+\cdots+l_{v+1} = v+1-u+k_0+\cdots+k_u$.  Since $l_{v+1}=k_u$, the two signs above are equal.  The $u=0$ must be treated separately.  In this case we compare instead \eqref{eq:diff1:af} and \eqref{eq:diff2:bf}, observing that
\[
(\psi_{\YB}\circ \b)\cdot[f_0,\ldots]'\cdot \psi_{\YB'}\inv = -\psi_\YB\cdot[\b\circ f_0,\ldots]'\cdot \psi_{\YB'}\inv.
\]
Here, we are using the bimodule structure described in Lemma \ref{lemma:bar complex 2}.

Terms of the form $[\ldots,f_u\circ \a_u,\ldots]'$ with $u\neq r+1$ occur in \eqref{eq:diff1:fa} and \eqref{eq:diff2:fa}, with signs
\[
(-1)^{u+k_0+\cdots+k_u} \text{ respectively } (-1)^{v+l_0+\cdots+l_v}.
\]
The signs are equal by Lemma \ref{lemma:signs oh my}, since $h_v=f_u$.   In the $u=r+1$ case we compare instead \eqref{eq:diff1:fa} and \eqref{eq:diff2:fb}, observing that 
\[
\psi_{\YB}\cdot[f_0,\ldots]'\cdot (\b'\circ \psi_{\YB'}\inv) = -\psi_\YB\cdot[\ldots, f_{r+1}\circ \b']'\cdot \psi_{\YB'}\inv.
\]

Terms of the form \eqref{eq:diff1:ff} occur in \eqref{eq:diff2:ff}, with signs
\[
(-1)^{u+k_0+\cdots+k_u} \text{ respectively } (-1)^{v+l_0+\cdots+l_v}.
\]
The signs are equal by Lemma \ref{lemma:signs oh my}, since $h_v=f_u$.  
\end{proof}

\begin{remark}\label{rmk:PhiTw sign}
In the ``unprimed'' convention for the bar complex (see Definition \ref{def:bar complex 2}) the formula for $\Xi_{\Tw}$ is
\begin{equation}\label{eq:expanded bar 2}
\Xi_{\Tw}([f_0,\ldots,f_{r+1}]):=\sum_{n_0,\ldots,n_r\in \Z_{\geq 0}} (-1)^{\sigma(l_0,\ldots,l_{s+1})-\sigma(k_0,\ldots,k_{r+1})}\psi_{\YB}\cdot [h_0,\ldots,h_{s+1}]\cdot \psi_{\YB'}\inv,
\end{equation}
where
\begin{equation}\label{eq:difference of sigmas}
\sigma(l_0,\ldots,l_{s+1})-\sigma(k_0,\ldots,k_{r+1})  = \sum_{v=0}^s(s+1-v)l_v - \sum_{u=0}^r (r+1-u)k_u.
\end{equation}
\end{remark}
The sign above can be rewritten explicitly in terms of $n_0,\ldots,n_r$ and the $k_u$:

\begin{lemma}\label{lemma:PhiTw sign}
Retain notation as above.  Then 
\begin{equation}\label{eq:sign rule for bar expansion}
\sum_{v=0}^s(s+1-v)l_v - \sum_{u=0}^r (r+1-u)k_u = \binom{n_0+\cdots+n_r+1}{2}+\sum_{u} (r-u)n_u+\sum_{u\leq u'}k_u n_{u'}
\end{equation}
\end{lemma}
\begin{proof}
Recall that $k_0,\ldots,k_{r+1}\in \Z/2$ and $n_0,\ldots,n_r$ are given, and fixed. Then we have $s=r+n_0+\cdots+n_r$ and
\[
l_v = \begin{cases}
k_u & \text{ if } v=u+n_0+\cdots+n_{u-1}\\
1 & \text{ if }  v-(u+n_0+\cdots+n_{u-1})\in \{1,\ldots,n_u\}
\end{cases}
\]
We say an index $0\leq v\leq s+1$ is of type $f$ if it satisfies the first condition above, and is of type $\a$ if it satisfies the second condition.  In the first case we have $s-v = r-u+n_u+\cdots+n_r$.  Thus, the ``type $f$'' terms contribute
\[
\Big((s+1-v)-(r+1-u)\Big)k_u = (n_u+\cdots+n_r)k_u
\] 
to \eqref{eq:difference of sigmas}.

If $v$ is of type $\a$, then we may write $v = u+n_0+\cdots+n_{u-1} + \a$ for some $a\in\{1,\ldots,n_u\}$, and its contribution to \eqref{eq:difference of sigmas} is
\[
(s+1-v)(1) = r-u+n_u+\cdots+n_r-(a-1).
\]
The sum of all of these (for fixed $u$) is $(r-u+n_u+\cdots+n_r)n_u-\binom{n_u}{2}$.  Adding up over all $u$ gives the overall sign
\[
\sum_{0\leq u\leq r} (r-u)n_u+\sum_{0\leq u\leq u'\leq r} n_u n_{u'}-\sum_{0\leq u\leq u'\leq r} \binom{n_u}{2}+\sum_{u\leq u'}k_u n_{u'}
\]
We can simplify further since
\[
\sum_{0\leq u\leq u'\leq r} n_u n_{u'}-\sum_{0\leq u\leq u'\leq r} \binom{n_u}{2}
=\frac{1}{2}\left(\sum_{u,u'} n_u n_{u'}+\sum_u n_u^2+\sum_u n_u \right) = \binom{n_0+\cdots+n_r+1}{2}
\]
This leaves us with the final formula \eqref{eq:sign rule for bar expansion}.
\end{proof}

\begin{example}
The case $r=0$ gives
\[
\Xi([f_0,f_1])= \sum_{n_0} (-1)^{\binom{n_0+1}{2}+k_0n_0} [f_0,\a_0,\ldots,\a_0,f_1].
\]
This case $r=1$ give
\[
\Xi([f_0,f_1,f_2]) = \sum_{n_0,n_1}(-1)^{\binom{n_0+n_1}{2} + n_0+k_0(n_0+n_1)+k_1n_1}[f_0,\a_0,\ldots,\a_0,f_1,\a_1,\ldots,\a_1,f_2]
\]
\end{example}

\subsection{Bar and $\pretr$}
Now we specialize to the case of $\EB=\pretr$.   Suppose we have a sequence of composable morphisms in $\pretr(\CS)$ as in \eqref{eq:sequenceBold}, using the alternate notation for the extremal terms \eqref{eq:extremal terms}.   For $0\leq u\leq r$ we will write
\[
\XB_u = \tw_{\boldsymbol{\a}_u}\left(\bigoplus_{a\in A_u} q^{i_a} X_a\right)
\]
where each $A_u$ is a finite set and the Maurer-Cartan element $\boldsymbol{\a}_u$ is strictly lower triangular with respect to a given partial order on $A_u$.

For each $a\in A_u$, let $\sigma_a\colon X_a\hookrightarrow \XB_u$ and $\pi_a\colon \XB_u\twoheadrightarrow X_a$ be the inclusion and projection of the direct summand.  Note that $\sigma_a$ and $\pi_a$ need not be degree zero, nor closed:
\[
|\sigma_a| = i_a = -|\pi_a| \, , \qquad d(\sigma_a) = \sum_{a'> a} \sigma_{a'}\circ \a_{a',a}\, , \qquad d(\pi_a) = (-1)^{1+\ip{\iota,a_a}} \sum_{a'<a} \a_{a,a'} \circ \pi_{a'}.
\]

\begin{lemma}\label{lemma:PhiPretr}
We have a chain map $\Xi_{\pretr}\colon \Bar(\pretr(\CS))\to \pretr(\CS)\otimes_\CS\Bar(\CS)\otimes_\CS\pretr(\CS)$ defined by
\begin{equation}\label{eq:PhiPretr gfg}
\Xi_{\pretr}([\id_{\XB_0},\mathbf{f}_1,\ldots,\mathbf{f}_{r},\id_{\XB_r}])=\sum_{\vec{n},\vec{a}}(-1)^{\sgn} 
\sigma_{a_0}\cdot \Big[\id_{X_{a_0}},h_1,\ldots,h_s, \id_{X_{a_s}}\Big]\cdot \pi_{a_s}
%
\end{equation}
where:
\begin{enumerate}
\item  the sum is over pairs $(\vec{n},\vec{a})$ with $\vec{n}=(n_0,\ldots,n_r)\in \Z_{\geq 0}^{r+1}$ and $\vec{a}\in A_0^{n_0+1}\times\cdots\times A_r^{n_r+1}$.
\item $s=r+n_0+\cdots+n_r$
\item $\displaystyle
h_v := \begin{cases}
\pi_{a_{v-1}}\circ \mathbf{f}_u\circ \sigma_{a_v}  & \text{ if } v=u+n_0+\cdots+n_{u-1}\\
\pi_{a_{v-1}}\circ \boldsymbol{\a}_u\circ \sigma_{a_v} & \text{ if } v-(u+n_0+\cdots+n_{u-1})\in \{1,\ldots,n_u\}
\end{cases}
$
\item the sign is
\[
\sgn = (r+N)\ip{\iota,i_{a_0}}+\binom{N+1}{2}+\sum_{0\leq u\leq r} (r-u)n_u+\sum_{1\leq u\leq u'\leq r}\ip{\iota,|\mathbf{f}_u|}n_{u'},
\]
where $N:=n_0+\cdots+n_r$.
\end{enumerate}
This map $\Xi_{\pretr}$ is a homotopy equivalence in $\Bim_{\pretr(\CS),\pretr(\CS)}$.
\end{lemma}
\begin{proof}
First, note that $\Xi_{\pretr}$ is well-defined (i.e.~the sum in \eqref{eq:PhiPretr gfg} is finite) since each $A_u$ is finite and $\pi_{a}\circ \boldsymbol{\a}_u\circ \sigma_{a'}=0$ unless $a>a'$ in the partial order on $A_u$, by hypothesis. 

The signs in the formula for $\Xi_{\pretr}$ come from the sign on $\Xi_{\SB}$ and $\Xi_{\Tw}$ (see Remark \ref{rmk:PhiTw sign} and Lemma \ref{lemma:PhiTw sign}).  The fact that $\Xi_{\pretr}$ is closed is a consequence of $\Xi_{\SB}$, $\Xi_{\AB}$, and $\Xi_{\Tw}$ being closed, and the fact that $\Xi_{\pretr}$ is a homotopy equivalence follows from its compatibility with counits, which is clear (see Proposition \ref{prop:bar props}).
\end{proof}

\section{Morita theory for dg categories}
\label{s:morita}
In this section we will assume that all dg categories have hom complexes which are projective over $\k$.

The bar complex of $\CS$ gives a projective resolution of the identity bimodule $\CS$, hence is useful when discussing (derived) Morita theory for dg categories.  There are many places where having a projective resolution of the identity bimodule of $\CS$ is handy:
\begin{enumerate}
\item Calculating Hochschild homology and cohomology.
\item Constructing adjoints to functors out of $\CS$.
\item Constructing homotopy colimits indexed by $\CS$.
\item Derived Morita theory
\end{enumerate}
In this paper, we focus only last of these, in the interest of space and keeping background to a minimum.  Derived Morita theory for dg categories was studied by To\"en \cite{ToenMorita}, with a the homotopy category of dg categories, with respect to a suitable model category structure. 

\subsection{The exact Morita 2-category}
\label{ss:exact bimod}
Recall that we assume all dg categories in this section have hom complexes which are projective over $\k$.
\begin{definition}\label{def:exact etc}
For a bimodule $B\in \Bim_{\CS,\DS}$ we say
\begin{enumerate}
\item $B$ is \emph{acyclic} if $\braCket{X}{B}{Y}\simeq 0$ for all $X,Y$.
\item $B$ is \emph{projective} if $\Hom_{\Bim}(B,Z)\simeq 0$ for all acyclic bimodules $Z\in \Bim_{\CS,\DS}$.
\item $B$ is \emph{left exact} if $B\ket{Y}\in\Bim_{\CS,\k}$ is projective for all $Y$.
\item $B$ is \emph{right exact} if $\bra{X}B\in \Bim_{\k,\DS}$ is projective for all $X$.
\item $B$ is \emph{exact} if it is both right and left exact.
\end{enumerate}
Given dg categories $\CS,\DS$, we let $\SMor_{\CS,\DS}\subset \Bim_{\CS,\DS}$ denote the full subcategory consisting of exact bimodules.  We define the \emph{exact (or sweet) Morita 2-category}, denoted $\SMor$, to be the 2-category with objects dg categories and 1-morphisms the dg categories $\SMor_{\CS,\DS}$ consisting of exact bimodules.
%
\end{definition}

\begin{remark}
Since $\Bar(\k)\simeq \k$ it is clear that a right or left module $M\in \Bim_{\k,\CS}$ (resp.~$M\in \Bim_{\CS,\k}$ is exact iff it is projective.
\end{remark}

\begin{remark}
The notion of exactness above should be likened to the notion of a ``sweet bimodule'' in \cite{KhovanovTangle}, but without any assumptions of finite generation.  For this reason, we may sometimes refer to $\SMor$ as the sweet Morita 2-category (or perhaps the s'Morita 2-category).
\end{remark}

\begin{remark}
Technically, we ought to say that $\SMor$ is a locally dg, non-strict 2-category.  The ``locally dg'' descriptor indicates that the 1-morphism categories in $\SMor$ are dg categories, and the ``non-strict'' descriptor refers to the unit and associativity of composition of 1-morphisms being satisfied only up to isomorphism (with the isomorphisms being given by explicit closed degree zero 2-morphisms satisfying the usual constraints).
\end{remark}

Next we give equivalent reformulations of the above properties using the bar complex.

\begin{proposition}\label{prop:acyclic projective exact}
Let $B\in \Bim_{\CS,\DS}$ be given.  Then
\begin{enumerate}
\item $B$ is acyclic iff $\Bar(\CS)\otimes_\CS B\otimes_\DS \Bar(\DS)\simeq 0$.
\item $B$ is projective iff $\Bar(\CS)\otimes_\CS B\otimes_\DS \Bar(\DS)\simeq B$.
\item $B$ is exact iff $\Bar(\CS)\otimes_\CS B\simeq B\otimes_\DS \Bar(\DS)$. 
\end{enumerate}
\end{proposition}
\begin{proof}
Suppose $B\in \Bim_{\CS,\DS}$ is acyclic.  Then $\bra{X}\otimes_\CS B\otimes_\DS \ket{Y}\simeq 0$ for all $X,Y$, which implies $\Bar(\CS)\otimes_\CS B\otimes_\DS \Bar(\DS)\simeq 0$.

For the converse, assume $\Bar(\CS)\otimes_\CS B\otimes_\DS \Bar(\DS)\simeq 0$.  Tensor on the left with $\bra{X}$, and on the right with $\ket{Y}$, and using the fact that $\bra{X}\otimes_\CS \Bar(\CS)\simeq \bra{X}$ and $\ket{Y}\otimes_\DS \Bar(\DS)\simeq \ket{Y}$, obtaining
\[
0 \simeq \bra{X}\otimes_\CS \Bar(\CS)\otimes_\CS B\otimes_\DS \Bar(\CS)\otimes_\CS \ket{Y} \simeq \braCket{X}{B}{Y},
\]
so $B$ is acyclic.  This proves (1). Statement (2) is a consequence of (1) and basic properties of categorical idempotents (exercise).

Finally, for (3) we observe that $B$ is left exact iff $B\otimes_\DS \Bar(\DS)$ is projective, which in light of (2) is equivalent to $\Bar(\CS)\otimes_\CS B\otimes_\DS \Bar(\DS)\simeq B\otimes_\DS \Bar(\DS)$.  Similarly, $B$ is right exact iff $\Bar(\CS)_\CS B\otimes_\DS \Bar(\DS)\simeq \Bar(\CS)\otimes_\CS B$.  Statement (3) follows immediately from this.
%
%
\end{proof}

\subsection{Quasi-isomorphism and derived Morita theory}
\label{ss:derived morita}
Recall again that we assume all dg categories in this section have hom complexes which are projective over $\k$.

\begin{definition}
Given exact bimodules $B,B'\in \SMor_{\CS,\DS}$, a closed degree zero map $f\colon B\to B'$ is a \emph{quasi-isomorphism} if the induced maps $\braCket{X}{B}{Y}\to \braCket{X}{B'}{Y}$ are homotopy equivalences for all $X,Y$.  We write $B\sim_{\mathrm{qiso}}B'$ if $B$ and $B'$ are related by a zig-zag of quasi-isomorphisms.
\end{definition}

The following is clear.
\begin{lemma}\label{lemma:bar detects qiso}
A bimodule map $f\colon B\to B'$ is a quasi-isomorphism iff $\Cone(f)$ is acyclic.  In particular $B\sim_{\mathrm{qiso}} B'$ if and only if $\Bar(\CS)\otimes_\CS B\otimes_\DS \Bar(\DS)\simeq \Bar(\CS)\otimes_\CS B'\otimes_\DS \Bar(\DS)$.\qed
\end{lemma}

%

%
%

%

\begin{remark}
The counit $\Bar(\CS)\to \CS$ is a quasi-isomorphism, hence $\CS\sim_{\mathrm{qiso}} \Bar(\CS)$.
\end{remark}

\begin{definition}\label{def:derived morita eq}
We say $\CS$ and $\DS$ are \emph{derived  Morita equivalent} if there exist bimodules $M\in \SMor_{\CS,\DS}$ and $N\in \SMor_{\DS,\CS}$ such that $M\otimes_\DS N\sim_{\mathrm{qiso}} \CS$ and $N\otimes_\CS M\sim_{\mathrm{qiso}}\DS$.  In this event we say that $M$ and $N$ are \emph{inverse derived Morita equivalences}.
\end{definition}

\begin{lemma}
Dg categories $\CS$ and $\DS$ are derived Morita equivalent if and only if there exist bimodules $M'\in \Bim_{\CS,\DS}$ and $N'\in \Bim_{\DS,\CS}$ such that $M'\otimes_\DS N'\simeq \Bar(\CS)$ and $N'\otimes_\CS M'\simeq \Bar(\DS)$.
\end{lemma}
\begin{proof}
Suppose first that $M\in \SMor_{\CS,\DS}$ and $N'\in \SMor_{\DS,\CS}$ are inverse derived Morita equivalences, and let $M':= \Bar(\CS)\otimes_\CS M$ and $N':=\Bar(\DS)\otimes_\DS N$.  Then
\begin{align*}
M'\otimes_\DS N' 
&\simeq \Bar(\CS)\otimes_\CS M\otimes_\DS \Bar(\DS) \otimes_\DS N\\
&\simeq \Bar(\CS)\otimes_\CS \Bar(\CS)\otimes_\CS M\otimes_\DS N\\
& \simeq \Bar(\CS)\otimes_\CS \CS\\
&\cong \Bar(\CS)
\end{align*}
where in the second lind we used that $M$ commutes with bar complexes, and in the third we used idempotence of the bar complex and Lemma \ref{lemma:bar detects qiso}.  A similar argument shows that $N'\otimes_\CS M'\simeq \Bar(\CS)$.  This proves the ``only if'' direction.

For the converse, suppose $M'\in \Bim_{\CS,\DS}$ and $N'\in \Bim_{\DS,\CS}$ are as in the statement.  Since $\Bar(\CS)\sim_{\mathrm{qiso}} \CS$ we need only show that $M'$ and $N'$ are exact.  For $M'$, we observe
\[
\Bar(\CS)\star M'\simeq M'\star N' \star M'\simeq M'\star \Bar(\DS)
\]
hence $M'$ is exact.  A similar argument shows that $N'$ is exact, and completes the proof.
\end{proof}

\subsection{A dg category is derived Morita equivalent to its envelopes}
\label{ss:morita and envelope}

\begin{theorem}
Let $\CS$ be a dg category whose hom complexes are projective over $\k$, and let $\EB$ be one of the envelope operations $\EB\in \{\SB,\AB,\pretr\}$.  Then $\CS$ is derived Morita equivalent to $\EB(\CS)$.
\end{theorem}
\begin{proof}
Let $M=\Bar(\CS)\otimes_\CS \EB(\CS)\in \Bim_{\CS, \EB(\CS)}$ and $N=\EB(\CS)|_{\CS}\in \Bim_{\EB(\CS),\CS}$ (the identity bimodule of $\EB(\CS)$, restricted to $\CS$ on one side).  Then
\[
M\otimes_{\EB(\CS)} N \cong \Bar(\CS)\otimes_\CS \EB(\CS)|_\CS \cong \Bar(\CS)
\]
where we are using that the double restriction ${}_\CS |\EB(\CS)|_\CS$ is isomorphic to $\CS$ since the inclusion $\CS\hookrightarrow \EB(\CS)$ is full.  In the other direction,
\[
N\otimes_\CS M \cong \EB(\CS)\otimes_\CS \Bar(\CS)\otimes_\CS \EB(\CS) \simeq \Bar(\EB(\CS))
\]
by Theorem \ref{thm:bar and generation}.
\end{proof}

\begin{example}\label{ex:Tw is not invt}
Here is a counterexample which shows that $\Tw(\CS)$ does not respect quasi-equivalence. Let $\CS$ be the dg category with one object $X$ and endomorphism algebra $\End_\CS(X)$ equal to the $2\times 2$ matrix algebra $M_{2\times 2}(\k)$ with $\Z$ gradings $\deg(e_{i,j})=i-j$ (where $e_{i,j}$ is the elementary matrix) and differential
\[
d(e_{i,j}) = [e_{2,1},e_{i,j}] = e_{2,1}e_{i,j} -(-1)^{i-j}e_{i,j}e_{2,1}.
\]
Then $X$ is contractible, with contracting homotopy $h=e_{1,2}$, but the twist $\tw_{-e_{2,1}}(X)$ is not contractible.  Indeed, the endomorphism algebra of $\tw_{-e_{2,1}}(X)$ is simply $M_{2\times 2}(\k)$ with the same grading as before, but zero differential.  Thus, $\CS$ is quasi-equivalent to the zero category, but $\Tw(\CS)$ is not.
\end{example}

\begin{remark}\label{rmk:Kar is not invt}
In this paper we have so far avoided discussing the Karoubi envelope (mostly because it is easy to include and there are no sign rules to worry about).  One may be tempted to conclude from Example \ref{ex:Tw is not invt} that the Karoubi envelope is also not a derived Morita invariant.  However, the error here is that the idempotent endomorphisms $e_{i,i}$ of $X$ are not closed, hence their ``images'' are not well defined objects of $\Kar(\CS)$.  In fact $\CS$ is already Karoubi complete.
\end{remark}

%

\appendix

\section{Counital idempotents}
\label{s:counital idempotents}
Some results on the bar complex rely on some general theory of (co)unital idempotents.  In the context of triangulated monoidal categories, (co)idempotents were studied by the author in \cite{HogancampIdempotent}.  We refer the reader to the introduction of \cite{HogancampIdempotent} for further references and context.  We will want lifts of this theory to the context of dg monoidal categories.  The main ideas are fairly straightforward consequences of results in \cite{HogancampIdempotent,HogancampConstructing}, but we include the details here for thoroughness.

This appendix is organized as follows.  In \S \ref{ss:counital objects} we review the definition and basic theory of counital objects (not necessarily idempotent) from \cite{HogancampConstructing}.  In \S \ref{ss:unital and counital idempts} we review (co)unital idempotents.  In \S \ref{ss:semiortho} we state and prove some semi-orthogonality properties of (co)unital idempotents and their complementary idempotents.  In \S \ref{ss:constructing} we recall a standard method for constructing counital idempotents (this is a dg lift of a construction in \cite{HogancampConstructing}), and finally \S \ref{ss:bar as idempt} establishes the application to bar complexes.

\subsection{Counital objects}
\label{ss:counital objects}

Assume first that $(\AS,\star,\one)$ is a $\k$-linear monoidal category. One should keep in mind that if $\AS$ were a dg monoidal category, then the constructions and results of this section should be applied to $H^0(\AS)$. We will omit all occurences the unitor and associator isomorphisms, for brevity.  We recall some definitions and results from \cite{HogancampConstructing}.

Let $C\in \AS$ be given.  We say that $\e\colon C\to \one$ is a \emph{counit} if $\e\star \id_C$ and $\id_C\star \e$ are retracts, i.e.~they have right inverses.  A \emph{counital object} in $\AS$ is an object $C\in \AS$ equipped with a counit $\e\colon C\to \one$.

\begin{theorem}\label{thm:counit order}
Let $(C_1,\e_1)$, $(C_2,\e_2)$ be counital objects in $\AS$.  The following are equivalent:
\begin{enumerate}
\item[(1L)] $C_1$ is a retract of $X\star C_2$ for some $X\in \AS$.
\item[(1R)] $C_1$ is a retract of $C_2\star X$ for some $X\in \AS$.
\item[(2L)] the map $\id_{C_1}\star \e_2\colon C_2\star C_2\to C_1$ is a retraction.
\item[(2R)] the map $\e_2\star \id_{C_1} \colon C_2\star C_2\to C_1$ is a retraction.
\item[(3)] there exists a map $\theta\colon C_1\to C_2$ such that $\e_2\circ \theta = \e_1$.
\end{enumerate}
\end{theorem}
\begin{proof}
We will sketch the proof, referring to \cite[\S 2.2]{HogancampConstructing} for details.  The implication (2L)$\Leftarrow$(1L) is clear.  For the implication (2L)$\Leftarrow$(1L), observe that if $C_1$ is a retract of $X\star C_2$ then consider the diagram
\[
\begin{tikzpicture}
\node (a) at (0,0) {$X\star C_2\star C_2$};
\node (b) at (5,0) {$X\star C_2$};
\node (c) at (0,-2) {$C_1\star C_2$};
\node (d) at (5,-2) {$C_1$};
\path[-stealth,very thick]
(a) edge node[above] {$\id_X\star \id_{C_2}\star \e_2$} (b)
(c) edge node[above] {$\id_{C_1}\star \e_2$} (d)
(a) edge node[left] {$r\star \id_{C_2}$} (c)
(b) edge node[left] {$r$} (d);
\end{tikzpicture}
\]
where $r$ is a chosen retraction $X\star C_2\to C_1$.   Every arrow but the bottom is a retraction, hence the bottom is a retraction (a retract of a retraction is a retraction) with section given by the composition $r\star \id_{C_2}\circ (id_X\star \id_{C_2}\star \e_2)\inv\circ r\inv$ (where here the superscript means right inverse, i.e.~section).  This shows the equivalence of (L1) and (L2).  

Now for (2L)$\Rightarrow$(3), consider the diagram
\[
\begin{tikzpicture}
\node (a) at (0,2) {$C_1\star C_2$};
\node (b) at (3,2) {$C_2$};
\node (c) at (0,0) {$C_1$};
\node (d) at (3,0) {$\one$};
\path[-stealth,very thick]
(a) edge node[above] {$\e_1\star \id_{C_2}$} (b)
(c) edge node[left]  {$(\id_{C_1}\star \e_2)\inv$} (a)
(c) edge node[below] {$\e_1$} (d)
(b) edge node[right] {$\e_2$} (d);
\path[-stealth,dashed]
(c) edge node[shift={(-.2,.2)}] {$\theta$} (b);
\end{tikzpicture}
\]
where as before $(\id_{C_1}\star \e_2)\inv$ means right inverse, which exists since we assume (2L).  We define $\theta$ by commutativity of the top left triangle, and the relation $\e_2\circ \theta=\e_1$ follows by commutativity of the square.  Thus (2L) implies (3).  Now, assume (3).  To show (2L), consider the diagram
\[
\begin{tikzpicture}
\node (a) at (0,2) {$C_1\star C_1$};
\node (b) at (3,2) {$C_1\star C_2$};
\node (c) at (0,0) {$C_1$};
\node (d) at (3,0) {$C_1$};
\path[-stealth,very thick]
(a) edge node[above] {$\id_{C_1}\star \theta$} (b)
(c) edge node[left]  {$(\id_{C_1}\star \e_1)\inv$} (a)
(c) edge node[below] {$\id_{C_1}$} (d)
(b) edge node[right] {$\id_{C_1}\star \e_2$} (d);
\path[-stealth,dashed]
(c) edge node[shift={(-.2,.2)}] {$\sigma$} (b);
\end{tikzpicture}
\]
where we are defining $\sigma$ by commutativity of the top left triangle.  The commutativity of the square tells us that $\sigma$ is a section of $\id_{C_1}\star \e_2$, hence (3) implies (2L).  This completes the proof that (1L), (2L), and (3) are equivalent.  The equivalence of (1R), (2R), (3) follows by symmetry.
\end{proof}

\begin{definition}\label{def:counit order}
Given counital objects $C_1,C_2\in \AS$ we write $C_1\leq C_2$ if either of the any conditions of Theorem \ref{thm:counit order} are satisfied.
\end{definition}

\begin{remark}
Note that $C_1\leq C_2$ depends only on the objects $C_i$ and not the specific choice of counits, by condition (1L) from Theorem \ref{thm:counit order}.
\end{remark}

\subsection{Unital and counital idempotents}
\label{ss:unital and counital idempts}
In this section we let $\AS$ denote a dg monoidal category.  

\begin{definition}\label{def:idempotent}
A \emph{counital idempotent} in $\AS$ is an object $P\in \AS$ equipped with a closed degree zero morphism $\e\colon P\to \one$ such that $\e\star \id_P$ and $\id_P\star \e$ are homotopy equivalences $P^{\star 2}\to P$.  A \emph{unital idempotent} in $\AS$ is an object $Q\in \AS$ equipped with a closed degree zero morphism $\eta\colon \one\to Q$ such that $\eta\star \id_Q$ and $\id_Q\star \eta$ are homotopy equivalences $Q\to Q^{\star 2}$.
%
\end{definition}

\begin{proposition}\label{prop:complements}
The following are equivalent.
\begin{enumerate}
\item $P$ is a counital idempotent in $\AS$ and $Q\simeq \Cone(P\to \one)$.
\item $Q$ is a unital idempotent in $\AS$ and $P\simeq \mathrm{Cocone}(\one\to Q)$.
\item there is a homotopy equivalence $\one\simeq (P\to Q)$ and $P\star Q\simeq 0\simeq Q\star P$.
\end{enumerate}
\end{proposition}
\begin{proof}
Exercise.
\end{proof}

\begin{definition}\label{def:complements}
If any of the equivalent statements of Proposition \ref{prop:complements} holds, then we say that $Q$ and $P$ are complementary idempotents, and write $Q\simeq P^c$ and $P\simeq Q^c$.
\end{definition}

Many properties of counital idempotents may be transferred to their complements, and vice versa.

\begin{lemma}\label{lemma:complement}
Below, let $P,P_1,P_2$ be arbitrary counital idempotents in $\AS$.  We  have the following elementary properties:
\begin{enumerate}
\item $(P^c)^c\simeq P$.
\item $P\star X\simeq X$ if and only if $P^c\star X\simeq 0$, with similar statements for $X$ on the left.
\item $P^c\star X\simeq X$ if and only if $P\star X\simeq 0$, with similar statements for $X$ on the left.
\item $P_1\simeq P_2$ if and only if $P_1^c\simeq P_2^c$.
\item $P_1\star P_2 \simeq P_2\star P_1$ if and only if $P_1^c\star P_2^c\simeq P_2^c\star P_1^c$.
\end{enumerate}
\end{lemma}
\begin{proof}
(1) follows from Gaussian elimination.  (2) and (3) are proven in \cite{}.  (4) is a consequence of (2) and (3).  Indeed $P\simeq P'$ implies $P\star P'\simeq P'$, which implies $P_1^c\star P_2\simeq 0$, which implies $P_1^c\star P_2^c\simeq P_1^c$.  By symmetry we have $P_1^c\star P_2^c\simeq P_2^c$ as well, hence $P_1^c\simeq P_2^c$.  The converse is proven similarly.
\end{proof}


\begin{theorem}\label{thm:counital idempotent order}
For counital idempotents $P_1,P_2\in \AS$ the following are equivalent.
\begin{enumerate}
\item $P_1\leq P_2$.
\item $P_1\star P_2\simeq P_2$.
\item $P_2\star P_1\simeq P_1$.
\end{enumerate}
\end{theorem}
\begin{proof}
If $P_1$ is a retract of $P_1\star P_2$ in $H^0(\AS)$, then $P_1\star P_2^c$ is a retract of $P_1\star P_2\star P_2^c\simeq 0$ in $H^0(\AS)$, hence $P_1\star P_2^c\simeq 0$, hence $P_1\star P_2\simeq P_1$.  Similarly, if $P_1$ is a retract of $P_2\star P_1$ then in fact $P_2\star P_1\simeq P_1$.  Now this theorem follows from Theorem \ref{thm:counit order}.\end{proof}

%
%

\subsection{Semi-orthogonality}
\label{ss:semiortho}
In this section we fix $\AS$, a dg monoidal category, and $\MS$, a dg category with a dg functor $\MS\otimes \AS\buildrel\star\over \to \MS$ making $\MS$ into a \emph{right $\AS$-module category}  (as usual, the associativity and unit isomorphisms here should satisfy some appropriate coherence conditions, and should be included as part of the structure of an $\AS$-module category).

We will denote the image of an object $M\in \MS$ by $[M]$.  We say $M$ is a \emph{homotopy retract of $M'$} if $[M]$ is a retract of $[M']$ in $H^0(\MS)$.

If $\e\colon P\to \one$ is a counital idempotent in $\AS$ then we let $\im(P),\ker(P)\subset \MS$ denote the full subcategories of $\MS$ consisting of objects $M\in \MS$ such that $M\star P\simeq M$, respectively $M\star P\simeq 0$.

\begin{lemma}\label{lemma:im P}
Let $\MS$ be a right $\AS$-module category and let $P$ be a counital idempotent in $\AS$, with counit $\e$.   For an object $M\in \MS$ the following are equivalent:
\begin{enumerate}
\item $M$ is a homotopy retract of an object in $\im(P)$.
\item $\id_M\star \e$ is a homotopy equivalence $M\star P\to M$.
\item $M\in \im(P)$.
\end{enumerate}
\end{lemma}
\begin{proof}
Suppose $M$ is a homotopy retract of $M'\in \im(P)$. Then $M\star P^c$ is a homotopy retract of $M'\star P^c$, which is contractible since $M'\in \im(P)$.  Since $P^c=\Cone(\e)$ this implies
\[
0\simeq M\star P^c = M\star \Cone(\e) \cong \Cone(\id_M\star \e),
\]
hence $\id_M\star \e$ is a homotopy equivalence.  This proves that (1) implies (2).  (2) obviously implies (3), which obviously implies (1).
\end{proof}

\begin{remark}
There is an obvious analogue of Lemma \ref{lemma:im P} for $\ker(P)$, which holds by swapping the roles of $P$ and $P^c$.
\end{remark}
%
%

The following lemma says that $\im(P)$ and $\ker(P)$ are \emph{semiorthogonal complements}.

\begin{lemma}\label{lemma:semiortho P and Pc}
Let $\MS$ be a right $\AS$-module category, and let $\e\colon P\to \one$ be a counital idempotent in $\AS$.  Then $\im(P)^\perp= \ker(P)$ and ${}^\perp\ker(P)=\im(P)$.  
\end{lemma}
Here we are using the notions of \emph{left} and \emph{right orthogonal}.  In elementary terms, $M\in \im(P)$ iff $\Hom_\MS(M,N)\simeq 0$ for all $N\in \ker(P)$, and $N\in \ker(P)$ iff $\Hom_\MS(M,N)\simeq 0$ for all $M\in \im(P)$.
\begin{proof}
Let $P^c$ be the complement of $P$ with unit $\eta\colon \one\to P^c$.  We claim that if $M\in \im(P)$ and $N\in \ker(P)$ then
\begin{equation}\label{eq:zero he}
\Hom_\MS(M,N)\to \Hom_\MS(M\star P,N\star Q)  \ , \qquad f\mapsto (\id_N\star \eta)\circ (f\star \id_\one)\circ (\id_M\star \e)
\end{equation}
is a homotopy equivalence. This follows since $\id_M\star \e$ and $\id_N\star \eta$ are homotopy equivalences by Lemma \ref{lemma:im P} (and the obvious dual version for unital idempotents).

Now, the map $\eta\circ \e$ is null-homotopic, being the composition of two consecutive maps in an exact triangle
\[
P\to \one \to \Cone(P\to \one).
\]
Thus, if $M\in \im(P)$ and $N\in \ker(P)$, then the homotopy equivalence \eqref{eq:zero he} is homotopic to zero, since it sends $f\mapsto f\star (\eta\circ \e)$; in particular $\Hom_\MS(M,N)$ is contractible.  This shows that $\im(P)\subset {}^\perp \ker(P)$ and $\ker(P)\subset \im(P)^\perp$.

For the opposite containment(s), suppose that $M\in {}^\perp\ker(P)$.  Then since $M\star P^c\in \ker(P)$ we have $\Hom_\MS(M,M\star P^c)\simeq 0$, in other words, we have a homotopy equivalence
\[
\Hom_\MS(M,M\star P)\buildrel\simeq\over\longrightarrow \End_\MS(M) \ , \qquad f\mapsto (\e\star \id_M)\circ f
\]
In particular, there exists a map $f\colon M\to M\star P$ such that $\e\star \id_M)\circ f \simeq \id_M$.    This means that $M$ is a homotopy retract of $M\star P$, hence $M\in \im(P)$ by Lemma \ref{lemma:im P}.  This finishes the proof that ${}^\perp\ker(P)= \im(P)$.  A similar argument finishes the proof that $\im(P)^\perp=\ker(P)$.
\end{proof}

The following is a dg analogue of the semi-orthogonality result from \cite{HogancampIdempotent}.  It is reminiscient of the well known fact that when computing $\Ext_R(M,N)$ between two $R$-modules $M,N$, one may resolve both $M,N$ by projective $R$-modules, or one may resolve only $M$.

\begin{proposition}\label{prop:semiortho P and A}
Let $\MS$ be a dg category on which $\AS$ acts on the right.  Let $\e\colon P\to \one$ be a counital idempotent in $\AS$.  If $M,N\in \MS$ are objects with $M\star P\simeq M$, then post-composing with $\id_N\star \e$ gives a homotopy equivalence
\begin{equation}\label{eq:postcompose}
\Hom_{\MS}(M, N\star P) \to \Hom_{\MS}(M, N).
\end{equation}
\end{proposition}
\begin{proof}
Linearity of the hom functor implies that the cone of the chain map \eqref{eq:postcompose} is isomorphic to $\Hom_\MS(M\star P,N\star P^c)$, which is contractible by Lemma \ref{lemma:semiortho P and Pc}.  Thus \eqref{eq:postcompose} is a homotopy equivalence.
\end{proof}

\begin{corollary}\label{cor:canonicalness}
If $P_1\leq P_2$ are counital idempotents in $\AS$, then the (closed, degree zero) map $\theta\colon P_1\to P_2$ with $\e_2\circ \theta\simeq \e_1$ is unique up to homotopy.
\end{corollary}
\begin{proof}
Proposition \ref{prop:semiortho P and A} (with $\MS=\AS$, $M=P_1$, $N=\one$) implies that post-composition with $\e_2$ defines a homotopy equivalence
\[
\Hom_\AS(P_1,P_2)\to \Hom_\AS(P_1,\one)
\] 
In particular $\e_1$ has a unique (up to homotopy) preimage $\theta$.
\end{proof}
\subsection{Constructing counital idempotents}
\label{ss:constructing}
We now construct infinite one-sided twisted complexes $\PB_C,\AB_C\in \pretr^{\amalg}(\AS)$ as follows.  We will shift $C$ and consider the shifted counit
\begin{equation}\label{eq:shifted counit}
\tilde{\e} := \e^0_{-\iota}:=\e\circ \phi_{C,-\iota}\inv \colon q^{-\iota} C\to \one.
\end{equation}
Now, set
\begin{equation}\label{eq:Ac}
\AB_C := \tw_\d\Big(\coprod_{n\geq 0} (q^{-\iota}C)^{\star n}\Big) \ ,
\end{equation}
where the twist $\d$ is $\d=\sum_{i,j\geq 0}\id^{\star i}\star \tilde{\e}\star \id^{\star j}$ (with no signs). Similarly let
\begin{equation}\label{eq:Pc}
\PB_C:= q^{\iota} \tw_\d\Big(\coprod_{n\geq 0} (q^{-\iota}C)^{\star n+1}\Big).
\end{equation}

\begin{remark}
The reader may be expecting an alternating sign in the definition of $\d$ above; this sign will appear once we unpack \eqref{eq:tensor sign} describing how $\star$ behaves on morphisms in $\SB(\AS)$.  Indeed, check that
\[
\e^0_{-\iota}\star \id^{-\iota}_{-\iota} = (\id\star \e)^{-\iota}_{-2\iota}
\]
(obtained by letting $k=l=i'=0$ and $i=j=j'=-\iota$ in \eqref{eq:tensor sign}), and
\[
\id^{-\iota}_{-\iota}\star \e^0_{-\iota} = -(\id\star \e)^{-\iota}_{-2\iota}
\]
(obtained by letting $k=l=j'=0$ and $i=i'=j=-\iota$ in \eqref{eq:tensor sign}).  More generally,
\[
(\id^{-\iota}_{-\iota})^{\star m}\star \e^0_{-\iota}\star (\id^{-\iota}_{-\iota})^{\star n} = (-1)^m (\id^{\star m}\star \e\star \id^{\star n})^{-(n+m)\iota}_{-(n+m+1)\iota}
\]
\end{remark}

\begin{remark}
The inclusion of the $n=0$ term defines a closed degree zero morphism $\one\to \AB_C$, and $\PB_C$ can be recovered as the cocone of this map:
\[
\PB_C \simeq \left(\one\to q^{\iota} \AB_C\right).
\]
\end{remark}

\begin{lemma}\label{lemma:A kills C}
We have $C\star \AB_C\simeq 0 \simeq \AB_C\star C$.
\end{lemma}
\begin{proof}
Recall that $q^{-\iota}C:=q^{-\iota}C$.  We will show that $q^{-\iota}C\star \AB_C\simeq 0$.  The proof that $\AB_C\star q^{-\iota}C\simeq 0$ is similar.  Let $\tilde{\Delta}_R\in \Hom^{-\iota}(q^{-\iota}C,q^{-\iota}C^{\star 2})$ and $h\in \End^{-\iota}(q^{-\iota}C)$ such that
\[
d(h) = \id_{q^{-\iota}C} - (\id_{q^{-\iota}C}\star \tilde{\e}) \circ \tilde{\Delta}_R
\]
which exist since $C$ is counital.  Then define $H\in \End^{-\iota}(q^{-\iota}C\star \AB_C)$ componentwise to 
be the sum of terms of the form $h\star \id^{\star n}\colon q^{-\iota}C^{\star n+1}\to q^{-\iota}C^{\star n+1}$ 
and $\tilde{\Delta}_R\star \id^{\star n}\colon q^{-\iota}C^{\star n+1}\to q^{-\iota}C^{\star n+2}$.  Let $\d$ be 
the as in \eqref{eq:Ac}.  Consider now the expression
\begin{equation}\label{eq:dH for CA}
d(H) + \d\circ H + H\circ \d,
\end{equation}
where $d$ indicates the differential on the (untwisted) coproduct $\coprod_{n\geq 0} q^{-\iota}C^{\star n+1}$ and $\d$ is the twist on this coproduct which defines $q^{-\iota}C\star \AB_C$.  That is to say, $\d$ is the sum of terms of the form $\id^{\star m+1}\star \tilde{\e}\star \id^{\star n}$ with $m,n\geq 0$.

Observe that all summands of $\d\circ H$ of the form $(\id^{\star m+2}\star \tilde{\e}\star \id^{\star n})\circ (\tilde{\Delta}_R\star \id^{\star m+n})$ (with $m,n\geq 0$) will cancel with the summands of $H\circ \d$ of the form $(\tilde{\Delta}_R\star \id^{\star m+n})\circ (\id^{\star m+1}\star \tilde{\e}\star \id^{\star n})$, via the super interchange law $(f\star \id)\circ (\id\star g) = (-1)^{\ip{|f|,|g|}}(\id\star g)\circ (f\star \id)$.  Thus $\d\circ H+H\circ \d$ is a sum of terms of the form $(\id^{\star 1}\star \tilde{\e}\star \id^{\star n})\circ (\tilde{\Delta}_R\star \id^{\star m+n})$.

Finally, $d(H)$ is $\id$ minus a sum of terms of the form $\id\star \tilde{\e}\star \id^{\star n}$.  This shows that \eqref{eq:dH for CA} equals $\id_{q^{-\iota}C\star \AB_C}$, hence $q^{-\iota}C\star \AB_C\simeq 0$.
\end{proof}

\begin{corollary}\label{cor:PC is idempt}
For any counital object $C\in \AS$, we have that $\PB_C$ is a counital idempotent with complement $\AB_C$.
\end{corollary}

\subsection{The bar complex as a counital idempotent}
\label{ss:bar as idempt}
The main property that we will use is that $\Bar(\CS)$ is a counital idempotent in the dg monoidal category $\Bim_{\CS,\CS}$. In fact, $\Bar(\CS)$ is a special case of the construction $\PB_C$ from \S \ref{ss:constructing}.   To see this, we first observe that since $\Bim_{\CS,\CS}$ is closed under coproducts, shifts, and twists, we have a canonical functor
\begin{equation}\label{eq:tot}
\Tot\colon \pretr^{\amalg}(\Bim_{\CS,\CS})\rightarrow \Bim_{\CS,\CS}
\end{equation}
which realizes that $\Bim_{\CS,\CS}$ is closed under taking one-sided twisted complexes.  It is a standard fact that tensor product $\otimes_\CS$ commutes with coproducts, so $\Tot$ is monoidal.  To apply the constructions from \S \ref{ss:constructing} we need a supply of counital objects in $\Bim_{\CS,\CS}$.

\begin{lemma}\label{lemma:CDC is coalg}
For any subcategory $\DS\subset \CS$, the bimodule $\CS\otimes_\DS \CS$ is a dg coalgebra in $\Bim_{\CS,\CS}$.  If $\DS\subset \ES\subset \CS$, then $\CS\otimes_\DS \CS\leq  \CS\otimes_\ES \CS$ (recall Definition \ref{def:counit order}.
\end{lemma}
\begin{proof}
A typical element in $\CS\otimes_\DS \CS$ is of the form $f\otimes f' \in \braket{X}{Y}\otimes \braket{Y}{X'}$ with $X,X'\in \Obj(\CS)$ and $Y\in \Obj(\DS)$.  Tensoring over $\DS$ means that the generators $f\otimes f'$ are to be regarded modulo relations of the form $(f\circ g)\otimes f' \sim f\otimes (g\circ f')$ for all morphisms $f\in \bra{X}\CS\ket{Y}$, $g\in \bra{Y}\DS\ket{Y'}$, $f'\in \bra{Y'}\CS\ket{X'}$.

The counit and comultiplication of $\CS\otimes_\DS \CS$ are defined by
\[
\begin{tikzpicture}
\node (a) at (0,0) {$\braket{X}{Y}\otimes \braket{Y}{X'}$};
\node (b) at (3,0) {$\braket{X}{X'}$};
\node (c) at (0,-1) {$f\otimes f'$};
\node (d) at (3,-1) {$f\circ f'$};
\path[-stealth]
(a) edge node[above] {$\e$} (b);
\path[|-stealth]
(c) edge node {} (d);
\end{tikzpicture}
\qquad \qquad
\begin{tikzpicture}
\node (a) at (0,0) {$\braket{X}{Y}\otimes \braket{Y}{X'}$};
\node (b) at (4.5,0) {$\braket{X}{Y}\otimes \braket{Y}{Y}\otimes \braket{Y}{X'}$};
\node (c) at (0,-1) {$f\otimes f'$};
\node (d) at (4.5,-1) {$f\otimes\id_Y\otimes f'$};
\path[-stealth]
(a) edge node[above] {$\Delta$} (b);
\path[|-stealth]
(c) edge node {} (d);
\end{tikzpicture}
\]
The coassociativity and counit axioms are easily verified.   This proves the first statement.

For the second, assume that $\DS$ is a subcategory of $\ES$, which is a subcategory of $\CS$.  Then clearly the quotient $\CS\otimes_\DS \CS\to \CS$ factors as a composition of quotient maps $\CS\otimes_\DS \CS\to \CS\otimes_\ES \CS\to \CS$, which shows $\CS\otimes_\DS \CS\leq  \CS\otimes_\ES \CS$ as claimed.
\end{proof}

There are a few important special cases of $\CS\otimes_\DS \CS$ that are worth pointing out.
\begin{example}\label{ex:YY coalg}
If $\DS$ has a single object $\Obj(\DS)=\{Y\}$, with $\Hom_\DS(Y,Y)=\k\, \id_Y$, then
\[
\CS\otimes_\DS\CS = \ketbra{Y}{Y}.
\]
\end{example}
\begin{example}\label{ex:YY coalg 2}
Generalizing Example \ref{ex:YY coalg}, if all morphisms in $\DS$ are scalar multiples of identity maps, then
\[
\CS\otimes_\DS\CS = \bigoplus_{Y\in \Obj(\DS)}\ketbra{Y}{Y}.
\]
\end{example}

We now have the relation to bar complexes.

\begin{lemma}\label{lemma:bar as idempt}
Let $\XS$ be a subset of objects of $\CS$, which we may view as a subcategory of $\CS$ in which all morphisms are scalar multiples of the identity maps. Then
\[
\Bar(\CS,\XS) = \PB_{\CS\otimes_{\XS}\CS}.
\]
\end{lemma}
\begin{proof}
An exercise in unpacking the definitions.
\end{proof}

Now, some basic theory of counital idempotents yields the following result on bar complexes.

\begin{proposition}\label{prop:bar props}
We have
\begin{enumerate}
\item If $\XS\subset \YS$ then $\Bar(\CS,\XS)\leq \Bar(\CS,\YS)$.
\item If $\XS\subset \YS$ and $\Xi\colon \Bar(\CS,\YS)\to \Bar(\CS,\XS)$ is a closed degree zero map of bimodules such that $\e_{\YS}\circ \Xi \simeq \e_{\XS}$ (where $\e_\XS$ and $\e_\YS$ are the appropriate counit maps), then $\Xi$ is a homotopy equivalence.
\end{enumerate}
\end{proposition}
\begin{proof}
Statement (1) is immediate from the relation $\CS\otimes_\XS\CS\leq \CS\otimes_\YS\CS$ from Lemma \ref{lemma:CDC is coalg}.  Statement (2) is a consequence of Theorem \ref{thm:counital idempotent order}: if $\Bar(\CS,\XS)\leq \Bar(\CS,\YS)$ and $\Bar(\CS,\YS)\leq \Bar(\CS,\XS)$, then $\Bar(\CS,\XS)\simeq \Bar(\CS,\YS)$.  Corollary \ref{cor:canonicalness} implies that $\Xi$ is in fact a homotopy equivalence.
\end{proof}

\bibliographystyle{alpha}
\bibliography{bib}

\end{document}